\documentclass[reqno]{amsart}

\calclayout

\usepackage[unicode]{hyperref}
\usepackage[lite]{amsrefs}
\usepackage{amsmath,amssymb}
\let\zeroslash\emptyset
\usepackage{xspace}

\def\paragraph{\subsection}
\makeatletter
\def\@sect#1#2#3#4#5#6[#7]#8{%
  \edef\@toclevel{\ifnum#2=\@m 0\else\number#2\fi}%
  \ifnum #2>\c@secnumdepth \let\@secnumber\@empty
  \else \@xp\let\@xp\@secnumber\csname the#1\endcsname\fi
  \@tempskipa #5\relax
  \ifnum #2>\c@secnumdepth
    \let\@svsec\@empty
  \else
    \refstepcounter{#1}%
    \edef\@secnumpunct{%
      \ifdim\@tempskipa>\z@ 
        \@ifnotempty{#8}{.\@nx\enspace}%
      \else
        \@ifempty{#8}{.}{.\@nx\enspace}%
      \fi
    }%
      \ifnum #2=\tw@ \def\@secnumfont{\bfseries}\fi{}%
    \protected@edef\@svsec{%
      \ifnum#2<\@m
        \@ifundefined{#1name}{}{%
          \ignorespaces\csname #1name\endcsname\space
        }%
      \fi
      \@seccntformat{#1}%
    }%
  \fi
  \ifdim \@tempskipa>\z@ 
    \begingroup #6\relax
    \@hangfrom{\hskip #3\relax\@svsec}{\interlinepenalty\@M #8\par}%
    \endgroup
    \ifnum#2>\@m \else \@tocwrite{#1}{#8}\fi
  \else
  \def\@svsechd{#6\hskip #3\@svsec
    \@ifnotempty{#8}{\ignorespaces#8\unskip
       \@addpunct.}%
    \ifnum#2>\@m \else \@tocwrite{#1}{#8}\fi
  }%
  \fi
  \global\@nobreaktrue
  \@xsect{#5}}
\makeatother

\makeatletter
\def\pxspace{\@ifnextchar.{\@}{.\@\xspace}}
\makeatother

\makeatletter
\def\citeHerr%
{\@ifnextchar*{\citeHerrs}{\cite{Herr1973a}\allowbreak\cite{Herr1973a-en}}}
\def\citeHerrs*#1#2{\cite{Herr1973a}*{#1}\allowbreak\cite{Herr1973a-en}*{#2}}
\makeatother

\makeatletter
\newcommand{\manuallabel}[2]{\edef\@currentlabel{#2}\label{#1}}
\makeatother
\newcommand{\itemlabel}[2]{\item[#1]\manuallabel{#2}{#1}}

\newcounter{theoremgroup}
\newenvironment{subnumbering}{%
  \global\let\thetheoremsave=\thetheorem
  \global\let\theHtheoremsave=\theHtheorem
  \setcounter{theoremgroup}{\value{theorem}}
  \setcounter{theorem}{0}
  \gdef\thetheorem{\thesection.\arabic{theoremgroup}.\arabic{theorem}}
  \gdef\theHtheorem{\theHsection.\arabic{theoremgroup}.\arabic{theorem}}
}{%
  \global\let\thetheorem=\thetheoremsave
  \global\let\theHtheorem=\theHtheoremsave
  \setcounter{theorem}{\value{theoremgroup}}
}

\newcommand{\hatzero}{\hat{0}}
\newcommand{\hatone}{\hat{1}}
\newcommand{\mvert}{\,|\,}      
\newcommand{\interject}[1]%
{\noalign{\begin{quote}#1\end{quote}}}
\newcommand{\leadingcolon}{\mathopen: \mathpunct{\hbox{}}}

\newtheorem{theorem}{Theorem}[section]

\newtheorem{definition}[theorem]{Definition}

\newtheorem{lemma}[theorem]{Lemma}

\newtheorem{remark}[theorem]{Remark}

\newtheorem{theoremcustomx}{Theorem}[theorem]
\newenvironment{theoremcustom}[1]{
  \renewcommand\thetheoremcustomx{#1}
  \theoremcustomx
}{\endtheoremcustomx}

\newcommand{\scrC}{\mathcal{C}}

\newcommand{\mbbN}{\mathbb{N}}

\newcommand{\mbbZ}{\mathbb{Z}}
\newcommand{\relgamma}{\mathrel{\gamma}}
\newcommand{\relxi}{\mathrel{\xi}}

\newcommand{\incomp}{\mathrel{||}}
\newcommand{\supdelta}{\textsuperscript{$\delta$}}
\newcommand{\supprime}{\textsuperscript{$\prime$}}

\newcommand{\lffc}{l.f.f.c\pxspace}
\newcommand{\mcs}{m.c.\ system\xspace}
\newcommand{\mcss}{m.c.\ systems\xspace}

\renewcommand{\P}[2]{{\phi_{#1}^{#2}}} 		
\newcommand{\PI}[2]{{{\phi_{#1}^{#2}}^{-1}}}       
\newcommand{\F}[2]{{F_{#1}^{#2}}}                 
\newcommand{\I}[2]{{I_{#1}^{#2}}}                 
\newcommand{\Z}[2]{{0_{#1}^{#2}}}                 
\renewcommand{\O}[2]{{1_{#1}^{#2}}}               
\newcommand{\PP}[2]{{\Phi_{#1}^{#2}}}             
\newcommand{\PPI}[2]{{\Psi_{#1}^{#2}}}            

\newcommand{\rmif}{\textup{if }}	

\newcommand{\disconnect}{\leavevmode\par}

\newcommand{\forwardref}[1]{\ref{#1}}

\begin{document}

\ifdefined\draftlabels
  \global\let\labelorig\label
  \global\def\label#1{[label:#1]\ \labelorig{#1}}
\fi

\title[Modular lattices --- gluing and dissection]%
{On the structure of modular lattices --- Unique gluing and dissection}
\author{Dale R. Worley}
\email{worley@alum.mit.edu}
\date{Mar 16, 2025} 

\begin{abstract}
This work proves that the process of gluing finite lattices to form
a larger lattice is bijective, that is each lattice is the glued sum
of a unique system of finite lattices, provided
the class of lattices is constrained to modular, locally-finite
lattices with finite covers.
The results of this work are not surprising given the prior
literature, but this seems to be the first proof that the processes of
gluing and dissection can be made inverses, and hence that gluing is
bijective.
\end{abstract}

\maketitle

\tableofcontents


\textit{v1 was published in 2025.
The changes in v2, published in 2025, are:  Statements that
depend the on monotony (axiom \forwardref{MCb2}) are marked with a dagger
($\dagger$).
Some proofs have been strengthened to no longer depend on monotony.
The numbering of statements is preserved between versions.
In some cases, a theorem \textup{xx.yy} in v1 has been split into two
theorems, \textup{xx.yy} and \textup{xx.yy.1}.}

\section{Introduction}

This work proves that the process of gluing lattices to form
a larger lattice can be made bijective, that is for each lattice there
is a unique (up to isomorphism) connected system whose sum is the
given lattice, provided
that the class of connected systems and the class of sum lattices are
suitably constrained.

The foundation work on lattice gluing is Herrmann \cite{Herr1973a}
(with English translation \cite{Herr1973a-en}).
Further development and exposition are in Day and
Herrmann \cite{DayHerr1988a}, Day and Frees \cite{DayFrees1990a},
and Haiman \cite{Haim1991a}*{sec.~1}.
The results of this work are not surprising given the prior
literature, but this seems to be the first proof that the processes of
gluing and dissection can be made inverses, and hence that gluing can be
made bijective.

Note that the theorems of \cite{Herr1973a} are stated only for
lattices of finite length.  We do not limit lattices to finite length.
Many of the proofs of \cite{Herr1973a} do not depend on finite length,
and generally those proofs have been copied or included by reference
here.

The
class of lattices we work with here is chosen for its utility within
the combinatorics of tableaux.
Specifically, we are interested in modular, \lffc (locally-finite and
with finite covers) lattices because they
are used in Fomin's growth diagram
construction.\cite{Fom1994a}\cite{Fom1995a}
The corresponding constraints on connected systems are that
(1) the skeleton lattice is \lffc;
(2) the blocks are finite, modular, and complemented; and
(3) the connected system is monotone.
(The monotone property is that no block is ``contained in'' a
block either above or below it, in a suitable sense.)

This work is a first step in the program
of \citeHerr*{sec.~0}{sec.~0},
``Als entscheidende Anwendung ergibt sich eine M\"oglichkeit, modulare
Verb\"ande endlicher L\"ange durch projektive Geometrien darzustellen ...''
(``A crucial application is representing finite-length modular lattices
using projective geometries ...''),
which is further elaborated in \cite{Wor2024c}.

The main result is:

\begin{theoremcustom}{\forwardref{th:inverses}}
The two mappings (1) gluing a \mcs to create a modular, \lffc lattice and
(2) dissecting a modular, \lffc lattice to create a \mcs are mutually
inverse mappings between
the category of isomorphism classes of \mcss and
the category of isomorphism classes of modular, \lffc lattices.
Thus, the two mappings are both bijective.
The the skeleton of the \mcs has a
minimum element iff the corresponding lattice has a minimum element.
\end{theoremcustom}

\section{Preliminaries}

\paragraph{Conventions for definitions}
In a sections we may define a symbol for an arbitrarily chosen
object of some class.  Consequently, some further definitions in the
section will be
implicitly parameterized by the chosen object because they explicitly depend
on the chosen object or on another object parameterized by the chosen object.
Similarly, when the constructions of such a section are ``imported''
into another section, a value for the chosen object will be fixed, and
not only the
symbol for that object becomes defined but implicitly all of the
definitions parameterized by that object as well.

\paragraph{Finite covers}

\begin{definition}
We define that
a lattice has \emph{finite covers} if every element in the
lattice has a finite number of upper and lower covers.
\end{definition}

\begin{definition} \label{def:lffc}
We define that
a lattice is \emph{locally finite with finite covers} ---
abbreviated \emph{\lffc} --- if it is both locally finite and has
finite covers (both upper and lower).
\end{definition}

\begin{remark}
\textup{
Note that finite covers is ``locally finite width''
property, but does not imply finite width.  For instance,
$\mbbZ \times \mbbZ$ is modular and \lffc, but all its ranks are
infinite.
}
\end{remark}

\begin{definition} \label{def:finitary}
We define that
a lattice is \emph{finitary} if all of its principal (downward) ideals
are finite.
\end{definition}

\begin{lemma}
If a lattice is finitary and has finite upper covers, then it is \lffc.
\end{lemma}

\paragraph{Tolerances}

\begin{definition} \cite{Band1981a}*{sec.~1}
We define that
a reflexive and symmetric binary relation $\relxi$ on a lattice $L$ is
a \emph{tolerance} if $\xi$ is compatible with the
meet and join of $L$, that is,
\begin{enumerate}
\item $a \relxi b$ and $c \relxi d$ imply $a \vee c \relxi b \vee d$, and
\item $a \relxi b$ and $c \relxi d$ imply $a \wedge c \relxi b \wedge d$.
\end{enumerate}
\end{definition}

\begin{lemma} \label{lem:tol} \cite{Band1981a}*{Lem.~1.1}
If $\relxi$ is a tolerance on a lattice $L$, then:
\begin{enumerate}
\item \label{lem:tol:i1} $x \relxi z$ and $x \leq y \leq z$ imply
$x \relxi y$ and $y \relxi z$.
\item \label{lem:tol:i2} $x \relxi y$ iff $x \vee y \relxi x \wedge y$,
\item \label{lem:tol:i3} $t \relxi x$, $t \relxi y$, and $t \leq x \wedge y$
imply $t \relxi x \vee y$.
\item \label{lem:tol:i4} $t \relxi x$, $t \relxi y$, and $t \geq x \vee y$
imply $t \relxi x \wedge y$.
\item \label{lem:tol:i5} $x \relxi x \vee y$ and $y \relxi x \vee y$
imply $x \vee y \relxi x \wedge y$, $x \relxi y$, and
$x \wedge y \relxi x, y$, that is, $\relxi$ holds among every pair of
$x$, $y$, $x \vee y$, and $x \wedge y$.
\item \label{lem:tol:i6} $x \relxi x \wedge y$ and $y \relxi x \wedge y$
imply $x \wedge y \relxi x \vee y$, $x \relxi y$, and
$x \vee y \relxi x, y$, that is, $\relxi$ holds among every pair of
$x$, $y$, $x \wedge y$, and $x \vee y$.
\end{enumerate}
\end{lemma}
\begin{proof}
\leavevmode

Regarding (\ref{lem:tol:i1}):
Since $x \relxi z$ and $y \relxi y$, $x \vee y \relxi z \vee y$, which
is equivalent to $y \relxi z$.
Similarly, $x \wedge y \relxi z \wedge y$, which
is equivalent to $x \relxi y$.

Regarding (\ref{lem:tol:i2}), (\ref{lem:tol:i3}), and (\ref{lem:tol:i4}):
see Bandelt \cite{Band1981a}*{Lem.~1.1(1--3)~Proof}.

Regarding (\ref{lem:tol:i5}):
By (\ref{lem:tol:i4}), $x \vee y \relxi x \wedge y$.
By (\ref{lem:tol:i2}), $x \relxi y$.
By (\ref{lem:tol:i1}), $x \wedge y \relxi x, y$.

Regarding (\ref{lem:tol:i6}):
Proved dually to (\ref{lem:tol:i5}).
\end{proof}

\begin{definition} \label{def:leq-gamma}
If $\relxi$ is a tolerance on a lattice $L$, for $x, y \in L$, we define
$x \leq_\xi y$ iff $x \leq y$ and $x \relxi y$.
We define $x \geq_\xi y$ iff $x \geq y$ and $x \relxi y$.
\end{definition}

Note that $\leq_\xi$ is \emph{not} transitive.

\paragraph{Properties of finite modular lattices}

\begin{theorem} \label{th:complemented}
If $M$ is a finite modular lattice, the following are equivalent:
\begin{enumerate}
\item \label{th:complemented:one} $\hatone$ is the join of atoms,
\item \label{th:complemented:zero} $\hatzero$ is the meet of coatoms,
\item \label{th:complemented:comp} $M$ is complemented,
\item \label{th:complemented:rc} $M$ relatively complemented,
\item \label{th:complemented:atom} $M$ is atomistic, and
\item \label{th:complemented:coatom} $M$ is coatomistic.
\end{enumerate}
\end{theorem}
\begin{proof}
\leavevmode

Regarding \ref{th:complemented:one} $\Rightarrow$ \ref{th:complemented:comp}:
\cite{Birk1967a}*{Th.~IV.6} Since $M$ is a finite modular lattice,
if $\hatone$ is the join of atoms, $M$ is complemented.

Regarding \ref{th:complemented:zero} $\Rightarrow$ \ref{th:complemented:comp}:
Dually to
\ref{th:complemented:one} $\Rightarrow$ \ref{th:complemented:comp},
if $\hatzero$ is the meet of coatoms,
the dual of $M$ is complemented.  But
the complemented property is self-dual, so $M$ is also complemented.

Regarding \ref{th:complemented:comp} $\Rightarrow$ \ref{th:complemented:rc}:
\cite{Birk1967a}*{Th.~I.14}
Any complemented modular lattice is relatively complemented.

Regarding \ref{th:complemented:comp} $\Rightarrow$ \ref{th:complemented:atom}:
\cite{Birk1967a}*{Th.~I.15~Cor.}
In a finite complemented modular lattice, every element
is the join of those atoms $\leq$ it.

Regarding \ref{th:complemented:comp} $\Rightarrow$ \ref{th:complemented:coatom}:
The complemented property is self-dual, so the dual of $M$ is also
complemented.
Thus in the dual of $M$, every element is the join of those atoms
$\leq$ it.
And so in $M$, every element is the meet of those coatoms $\geq$ it.

Regarding \ref{th:complemented:rc} $\Rightarrow$ \ref{th:complemented:comp},
\ref{th:complemented:atom} $\Rightarrow$ \ref{th:complemented:one}, and
\ref{th:complemented:coatom} $\Rightarrow$ \ref{th:complemented:zero}:
These are immediate.
\end{proof}

\section{Modular connected systems}

\newcommand{\flagstart}{\leavevmode\unskip}

In this section we define \emph{modular connected systems}, a variant
of the constructions
introduced by Herrmann in \citeHerr*{sec.~4}{sec.~4}
and Day and Herrmann in \cite{DayHerr1988a}*{Def.~4.7}.
We also define isomorphism
between modular connected systems.

\begin{definition} \label{def:mcs}
We define a \emph{modular connected system}%
\footnote{We dispense with the ubiquitous ``$S$-'' prefixes when
defining terms for our gluing construction.}
--- abbreviated a \emph{\mcs} --- to be comprised of:
\begin{enumerate}
\item[\textbullet] a \emph{skeleton} lattice $S$,
\item[\textbullet] an \emph{overlap tolerance} $\relgamma$, which is
a tolerance on $S$,
\item[\textbullet] a family of \emph{blocks} $(L_x)_{x \in S}$, which
are lattices, and
\item[\textbullet] a family of \emph{connections}
$(\P{x}{y})_{x, y \in S, x \leq_\gamma y}$, which are mappings,
\end{enumerate}
that obey these axioms:
\begin{enumerate}
\itemlabel{(MC1)}{MCf} The skeleton $S$ is \lffc.
\itemlabel{(MC2)}{MCg} The blocks $L_x$ are finite modular complemented lattices,
\itemlabel{(MC3)}{MCe}
Each $\P{x}{y}$ is a lattice isomorphism from a filter $\F{x}{y}$ of
$L_x$ to an ideal $\I{x}{y}$ of $L_y$.%
\footnote{Note that these are \emph{lattice} filters and ideals,
not \emph{poset} filters and ideals;
the filters are closed under meet and the ideals are closed under join.
Since each $L_\bullet$ is finite, the filters $\F{\bullet}{\bullet}$ and
ideals $\I{\bullet}{\bullet}$ are all principal.}
\itemlabel{(MC4)}{MCa} For any $x \in S$, $\F{x}{x} = \I{x}{x} = L_x$
and $\P{x}{x}$ is the identity map on $L_x$.
\itemlabel{(MC5)}{MCh} If $x \leq_\gamma y \leq_\gamma z$ in $S$ and
$\I{x}{y} \cap \F{y}{z} \neq \zeroslash$, then $x \relgamma z$.
\itemlabel{(MC6)}{MCc} For every $x \leq z \leq y$ in $S$ where
$x \relgamma y$, then
$\F{x}{y} = \PI{x}{z}(\I{x}{z} \cap \F{z}{y})$,
$\I{x}{y} = \P{z}{y}(\I{x}{z} \cap \F{z}{y})$, and
$\P{x}{y} = \P{z}{y} \circ \P{x}{z}|_{\F{x}{y}}$.%
\footnote{That is, $\P{x}{y} = \P{z}{y} \circ \P{x}{z}$ if
$\P{x}{z}$ is considered as a partial function
from $L_x$ to $L_z$ and $\P{z}{y}$ is considered as a partial
function from $L_z$ to $L_y$.}
\itemlabel{(MC7)}{MCd} For every $x, y \in S$ for which
$x \relgamma y$,%
\footnote{And thus $\relgamma$ holds among every pair of $x$, $y$, $x \vee y$,
and $x \wedge y$ by lem.~\ref{lem:tol}(\ref{lem:tol:i1})(\ref{lem:tol:i2}).}
\begin{enumerate}
\item $\I{x}{x \vee y} \cap \I{y}{x \vee y} \subset \I{x \wedge y}{x \vee y}$
and
\item $\F{x \wedge y}{x} \cap \F{x \wedge y}{y} \subset \F{x \wedge y}{x \vee y}$.
\end{enumerate}
\itemlabel{(MC8.1)}{MCb1} If $x \lessdot y$ in $S$, then
$x \relgamma y$ (and thus $\I{x}{y}$ and $\F{x}{y}$ exist).
\itemlabel{(MC8.2)}{MCb2} If $x \lessdot y$ in $S$
(and thus $\I{x}{y}$ and $\F{x}{y}$ exist),
$\F{x}{y} \neq L_x$, and $\I{x}{y} \neq L_y$.%
\footnote{Thus, the \mcs is monotone in the stronger,
self-dual sense.\cite{Herr1973a-en}*{Def.~T.1}}
\end{enumerate}
By abuse of language, we say that the blocks $L_x$ ``are a \mcs''
when the remaining parts of the \mcs are implicit.
\end{definition}

Note that axiom (MC8) in v1 of this article has been split
into two parts, \ref{MCb1} and \ref{MCb2} (monotony).  The results that
depend on monotony are marked with a dagger ($\dagger$).

Note that many of our symbols have both a subscript and a superscript that
are elements of $S$ --- $\P{x}{y}$, $\F{x}{y}$, $\I{x}{y}$, and others
--- In all such
cases, the symbol is defined only for $x \leq y$ (and usually only if
$x \leq_\gamma y$); the subscript is
the \emph{lower} element and the superscript is the \emph{upper}
element of $S$.  This convention is regardless of exactly how the symbol is
related to $x$ and $y$.  E.g.,\ $\F{x}{y} \subset L_x$ but
$\I{x}{y} \subset L_y$.

Including the overlap tolerance in the definition of a \mcs is annoying,
but the distinction between pairs of elements $x$ and $y$ in $S$ for
which $x \relgamma y$ and pairs for which $x \not\relgamma y$ seems to
be fundamental
--- all alternative definitions in \cite{DayHerr1988a} introduce the
same distinction in some other guise.%
\footnote{A parallel is the distinction in set theory between
properties of objects that can be collected into a set and those which
can only be collected into a proper class.}

We will show in th.~\forwardref{th:dissection-of-sum} that
monotony, \ref{MCb2}, is necessary to prove
there to be only one \mcs with a particular sum lattice.

\begin{remark}
\textup{Note that we do not require the skeleton to be modular.
In particular, \citeHerr*{Satz~7.2 and~7.3}{Th.~7.2 and~7.3} shows that
any \emph{finite} lattice is the dissection skeleton of some finite (and
hence \lffc) modular lattice.
It is an open question whether every \lffc lattice is the
skeleton of some modular, \lffc lattice.}

\textup{We are stricter than \cite{Herr1973a} in that we require the blocks to
be finite, whereas Herrmann only requires them to have finite height.
We also require them to be modular, which Herrmann treats
as an add-on property, and to be complemented.
We are also stricter in that we require the skeleton to be \lffc.
On the other hand, we are looser in that \cite{Herr1973a} requires the
skeleton to be finite height and we do not.}
\end{remark}

\begin{definition} \label{def:mcs-dual}
The \emph{dual} of a \mcs is defined as the construction:
\begin{enumerate}
\item a skeleton lattice $S^\delta$ (the dual lattice of $S$),
\item the overlap tolerance $\relgamma$, which operates on the
elements of $S^\delta$, which are the same as the elements of $S$,
\item the blocks $({L_x}^\delta)_{x \in S^\delta}$ (the dual lattices
of the $L_x$),
\item the connections $(\PI{y}{x})_{x, y \in S^\delta, x \leq_\gamma y}$
($x, y \in S^\delta, x \leq_\gamma y$ is equivalent to
($y, x \in S, y \leq_\gamma x$).
\end{enumerate}
\end{definition}

\begin{theorem} \label{th:mcs-dual}
The dual of a \mcs is a \mcs.
The dual of the dual of a \mcs is itself.
\end{theorem}

We will use this fact often in proofs to avoid providing dual
arguments for the duals of facts that we have proven.

\begin{definition} \label{def:mcs-iso}
We define an isomorphism $\chi$ between a \mcss
$\scrC$ (comprised of skeleton $S$, overlap tolerance $\relgamma$,
blocks $L_\bullet$, connections $\P{\bullet}{\bullet}$,
connection sources
$\F{\bullet}{\bullet}$, and connection targets
$\I{\bullet}{\bullet}$) and
a \mcs $\scrC^\prime$ (comprised of skeleton $S^\prime$, overlap
tolerance $\relgamma^\prime$,
blocks $L_\bullet^\prime$, connections
$\P{\bullet}{\bullet}^\prime$, connection sources
$\F{\bullet}{\bullet}^\prime$, and connection targets
$\I{\bullet}{\bullet}^\prime$) to be comprised of:
\begin{enumerate}
\item a lattice isomorphism $\chi_S$ from $S$ to $S^\prime$ and
\item a family of lattice isomorphisms $\chi_{Bx}$ for every $x \in S$,
with $\chi_{Bx}$ being an isomorphism from $L_x$ to
$L^\prime_{\chi_S(x)}$,
\end{enumerate}
for which:
\begin{enumerate}
\item for $x, y \in S$, $x \relgamma y$ iff
$\chi_S(x) \relgamma^\prime \chi_S(y)$,
\item for $x \leq_\gamma y \in S$,
$\F{\chi_S(x)}{\chi_S(y)}^\prime = \chi_{Bx}(\F{x}{y})$ and
$\I{\chi_S(x)}{\chi_S(y)}^\prime = \chi_{By}(\I{x}{y})$,
\item for $x \leq_\gamma y \in S$,
$\chi_{By} \circ \P{x}{y} =
\P{\chi_S(x)}{\chi_S(y)}^\prime \circ \chi_{Bx}|_{\F{x}{y}}$.%
\footnote{That is, for all $a \in \F{x}{y}$,
$\chi_{By}(\P{x}{y}(a)) = \P{\chi_S(x)}{\chi_S(y)}^\prime(\chi_{Bx}(a))$.}
\end{enumerate}
\end{definition}

\begin{theorem} \label{th:mcs-iso-equiv}
Isomorphism between \mcss is an equivalence relation.
\end{theorem}

\section{Gluing}
The gluing of a set of finite lattices into a sum lattice was
introduced in Herrmann \citeHerr*{sec.~1}{sec.~1}.

\paragraph{Basic properties of modular connected systems}

\begin{definition}
In this section, we define $\scrC$ to be an arbitrarily chosen \mcs
comprised of
skeleton $S$, overlap tolerance $\relgamma$,
blocks $L_\bullet$, connections $\P{\bullet}{\bullet}$,
connection sources
$\F{\bullet}{\bullet}$, and connection targets
$\I{\bullet}{\bullet}$).
\end{definition}

\begin{definition}
If $x \leq_\gamma y$,
we define $\Z{x}{y}$ to be the minimum of $\F{x}{y}$ and
$\O{x}{y}$ to be the maximum of $\I{x}{y}$.
\end{definition}

\begin{lemma}
If $x \leq_\gamma y$,
\begin{enumerate}
\item $\O{x}{y} = \P{x}{y}\, \hatone_{L_x}$ and
\item $\Z{y}{y} = \PI{x}{y}\, \hatzero_{L_y}$.
\end{enumerate}
\end{lemma}

Thus we can say that $\Z{x}{y}$ is $\hatzero_{L_y}$ ``transported
down'' to $L_x$ and $\O{x}{y}$ is $\hatone_{L_x}$ ``transported
up'' to $L_y$.

\begin{lemma}\flagstart$\dagger$ \label{lem:monotone-0-1}
If $x < y$ and $x \relgamma y$,
\begin{enumerate}
\item $\O{x}{y} < \hatone_{L_y}$, and
\item $\Z{x}{y} > \hatzero_{L_x}$.
\end{enumerate}
\end{lemma}
\begin{proof}
\leavevmode

Regarding $\O{x}{y} < \hatone_{L_y}$:
Of necessity, $\O{x}{y} \leq \hatone_{L_y}$.
But if $\O{x}{y} = \hatone_{L_y}$, then $\I{x}{y} = L_y$, which
violates \ref{MCb2}.
$\Z{x}{y} > \hatzero_{L_x}$ is proved dually.
\end{proof}


\begin{lemma}\citeHerr*{(6)}{(6)} \label{lem:(6)}
If $x, y, z \in S$, $x \leq z \leq y$, and $x \relgamma y$
(implying by lem.~\ref{lem:tol}(\ref{lem:tol:i1}), $x \relgamma z$ and
$z \relgamma y$),
then $\F{x}{y} \subset \F{x}{z}$ and $\I{x}{y} \subset \I{z}{y}$.
\end{lemma}
\begin{proof}
\leavevmode

Of necessity,
\begin{equation*}
\hatzero_{L_z} \leq \PI{z}{y}\, \hatzero_{L_y}
\end{equation*}
Then
\begin{alignat*}{2}
\Z{x}{z}
& = \PI{x}{z}\, \hatzero_{L_z} \\
& \leq \PI{x}{z}(\PI{z}{y}\, \hatzero_{L_y}) \\
\interject{by \ref{MCc},}
& = \PI{x}{y}\, \hatzero_{L_y} \\
& = \Z{x}{y}
\end{alignat*}
Which shows that $\F{x}{y} \subset \F{x}{z}$.

We prove $\I{x}{y} \subset \I{z}{y}$ dually.
\end{proof}

\ref{MCd} can be strengthened:

\begin{lemma}\citeHerr*{(7)}{(7)} \label{lem:(7)}
For any $x, y \in S$ for which $x \relgamma y$
(and thus $x \vee y \relgamma x \wedge y$),
\begin{enumerate}
\item $\I{x}{x \vee y} \cap \I{y}{x \vee y} = \I{x \wedge y}{x \vee y}$ and
\item $\F{x \wedge y}{x} \cap \F{x \wedge y}{y} = \F{x \wedge y}{x \vee y}$.
\end{enumerate}
\end{lemma}
\begin{proof}
\leavevmode

This is a direct consequence of \ref{MCd} and lem.~\ref{lem:(6)}.
\end{proof}

\begin{lemma} \label{lem:MCd-stronger}
If $z \leq_\gamma x, y$ in $S$, then
$\F{z}{x} \cap \F{z}{y} = \F{z}{x \vee y}$.
If $x, y \leq_\gamma z$ in $S$, then
$\I{x}{z} \cap \I{y}{z} = \I{x \wedge y}{z}$.
\end{lemma}
\begin{proof}
\leavevmode

Regarding $\F{z}{x} \cap \F{z}{y} = \F{z}{x \vee y}$:
\begin{alignat*}{2}
\interject{By \ref{MCc},}
\F{z}{x} \cap \F{z}{y}
& = \PI{z}{x \wedge y}(\I{z}{x \wedge y} \cap \F{x \wedge y}{x})
\cap \PI{z}{x \wedge y}(\I{z}{x \wedge y} \cap \F{x \wedge y}{y}) \\
\interject{because the $\PI{\bullet}{\bullet}$ are one-to-one,}
& = \PI{z}{x \wedge y}(\I{z}{x \wedge y} \cap \F{x \wedge y}{x}
\cap \F{x \wedge y}{y}) \\
\interject{by lem.~\ref{lem:(7)},}
& = \PI{z}{x \wedge y}(\I{z}{x \wedge y} \cap \F{x \wedge y}{x \vee y}) \\
\interject{by \ref{MCc} again,}
& = \F{z}{x \vee y}
\end{alignat*}

$\I{x}{z} \cap \I{y}{z} = \I{x \wedge y}{z}$ is proved dually.
\end{proof}


\begin{lemma} \label{lem:between-equal}
If $x, y \leq z \leq w$ in $S$, $x, y \relgamma w$, $a \in \F{x}{w}$,
$b \in \F{y}{w}$, and $\P{x}{w}\,a = \P{y}{w}\,b$,
then $\P{x}{z}\,a = \P{y}{z}\,b$.
\end{lemma}
\begin{proof}
\leavevmode

By \ref{MCc},
$$ \P{z}{w}(\P{x}{z}\,a) = \P{x}{w}\,a =
\P{y}{w}\,b = \P{z}{w}(\P{y}{z}\,b) $$
Since $\P{z}{w}$ is one-to-one, $\P{x}{z}\,a = \P{y}{z}\,b$.
\end{proof}

\begin{lemma} \label{lem:between-equal-dual}
If $w \leq z \leq x,y$ in $S$, $x, y \relgamma w$, $a \in \I{w}{x}$,
$b \in \I{w}{y}$, and $\PI{w}{x}\,a = \PI{w}{y}\,b$,
then $\PI{z}{x}\,a = \PI{z}{y}\,b$.
\end{lemma}
\begin{proof} \disconnect
This lemma is proved dually to lem.~\ref{lem:between-equal}.
\end{proof}

\paragraph{The sum}

\begin{lemma} \citeHerr*{Hilfs.~4.1}{Prop.~4.1} \label{lem:4.1}
Let $x, y \in S$, $a \in L_x$, and $b \in L_y$.
Then the following are equivalent:
\begin{enumerate}
\item There is a $z \in S$ with
$z \geq_\gamma x, y$; $a \in \F{x}{z}$; $b \in \F{y}{z}$;
and $\P{x}{z}\,a = \P{y}{z}\,b$.
\item $x \relgamma y$, $a \in \F{x}{x \vee y}$, $b \in \F{y}{x \vee y}$,
and $\P{x}{x \vee y}\,a = \P{y}{x \vee y}\,b$.
\item There is a $z \in S$ with
$z \leq_\gamma x, y$; $a \in \I{x}{z}$; $b \in \I{y}{z}$;
and $\PI{z}{x}\,a = \PI{z}{y}\,b$.
\item $x \relgamma y$, $a \in \I{x \wedge y}{x}$, $b \in \I{x \wedge y}{y}$,
and $\PI{x \wedge y}{x}\,a = \PI{x \vee y}{y}\,b$.
\end{enumerate}
\end{lemma}
\begin{proof}
\leavevmode

Regarding (2) $\Rightarrow$ (1) and (4) $\Rightarrow$ (3):
Lem.~\ref{lem:tol}(\ref{lem:tol:i1})(\ref{lem:tol:i2}) shows that
$x \relgamma y$ implies $x \vee y \relgamma x,y$ and
$x \wedge y \relgamma x,y$ and thus we can choose $z = x \vee y$ for
(1) or $z = z \wedge y$ for (3).

Regarding (1) $\Rightarrow$ (2):
Applying lem.~\ref{lem:between-equal} to
$\P{x}{z}\,a = \P{y}{z}\,b$ shows that
$\P{x}{x \vee y}\,a = \P{y}{x \vee y}\,b$.

Regarding (3) $\Rightarrow$ (4):
Applying lem.~\ref{lem:between-equal-dual} to
$\PI{z}{x}\,a = \PI{z}{y}\,b$ shows that
$\PI{x \vee y}{x}\,a = \PI{x \vee y}{y}\,b$.

Regarding (2) $\Rightarrow$ (4):
By \ref{MCd}(a), there is a
$c \in \F{x \wedge y}{x \vee y} \subset L_{x \wedge y}$ such that
$\P{x \wedge y}{x \vee y}\, c = \P{x}{x \vee y}\, a =
\P{y}{x \vee y}\, b$.
By lem.~\ref{lem:between-equal},
$\P{x \wedge y}{x}\, c = a$ and so
$c = \PI{x \wedge y}{x}\, a$.
Likewise, $c = \PI{x \wedge y}{y}\, b$, so
$\PI{x \wedge y}{x}\, a = \PI{x \wedge y}{y}\, b$.

The implication (4) $\Rightarrow$ (2) is proved dually
using \ref{MCd}(b) and lem.~\ref{lem:between-equal-dual}.
\end{proof}

\begin{definition} \label{def:sim}
Given a \mcs, we define $M = \{ (x, a) \mvert x \in S, a \in L_x \}$.
We define the relation $\sim$ on $M$ as:
$(x, a) \sim (y, b)$ iff $x \relgamma y$, $a \in \F{x}{x \vee y}$,
$b \in \F{y}{x \vee y}$, and
$\P{x}{x \vee y}\,a = \P{y}{x \vee y}\,b$.
\end{definition}

\begin{lemma} \label{lem:sim-dual}
Given $(x, a), (y, b) \in M$,
$(x, a) \sim (y, b)$ iff $x \relgamma y$, $a \in \I{x}{x \wedge y}$,
$b \in \I{y}{x \wedge y}$, and
$\PI{x \wedge y}{x}\,a = \PI{x \wedge y}{y}\,b$.
\end{lemma}
\begin{proof}
\leavevmode

This follows directly by applying lem.~\ref{lem:4.1} to
def.~\ref{def:sim}.
\end{proof}

\begin{lemma} \label{lem:sim-leq-z}
If $x \leq y$ in $S$:
\begin{enumerate}
\item $(x,a) \sim (y,b)$ iff $x \relgamma y$, $a \in \F{x}{y}$,
and $\P{x}{y}\,a = b$.
\item $(x,a) \sim (y,b)$ iff $x \relgamma y$, $b \in \I{x}{y}$,
and $\PI{x}{y}\,b = a$.
\end{enumerate}
\end{lemma}
\begin{proof}
\leavevmode

This follows directly by applying \ref{MCa} to def.~\ref{def:sim} and
lem.~\ref{lem:sim-dual}.
\end{proof}

\begin{theorem} \citeHerr*{(21)}{(21)} \label{th:equiv-rel}
$\sim$ is an equivalence relation on $M$.
\end{theorem}
\begin{proof}
\leavevmode

Regarding transitivity of $\sim$:
Let $(x_0, a_0) \sim (x_1, a_1)$ and $(x_1, a_1) \sim (x_2, a_2)$, so
\begin{gather*}
x_0 \vee x_1 \relgamma x_0, x_1
\textrm{ and }
x_1 \vee x_2 \relgamma x_1, x_2 \\
\P{x_0}{x_0 \vee x_1}\,a_0 = \P{x_1}{x_0 \vee x_1}\,a_1
\textrm{ and }
\P{x_1}{x_1 \vee x_2}\,a_1 = \P{x_2}{x_1 \vee x_2}\,a_2.
\end{gather*}
Define $x = (x_0 \vee x_1) \wedge (x_1 \vee x_2)$ and
$y = x_0 \vee x_1 \vee x_2$.
Since $x_1 \leq x \leq x_0 \vee x_1, x_1 \vee x_2$, we know
$x_1 \relgamma x$, $x \relgamma x_0 \vee x_1$,
$x \relgamma x_1 \vee x_2$, and by
lem.~\ref{lem:tol}(\ref{lem:tol:i6}),
$x \relgamma (x_0 \vee x_1) \vee (x_1 \vee x_2) = y$.
By \ref{MCc},
$a_1 \in \F{x_1}{x}$, $\P{x_1}{x}\,a_1 \in \F{x}{x_0 \vee x_1}$,
and $\P{x_1}{x}\,a_1 \in \F{x}{x_1 \vee x_2}$, and so
by \ref{MCd}, $\P{x_1}{x}\,a_1 \in \F{x}{y}$.
This means that
$\P{x_1}{x}\,a_1 \in \I{x_1}{x} \cap \F{x}{y} \neq \zeroslash$ and
by \ref{MCh}, $x_1 \relgamma y$ and then by
lem.~\ref{lem:tol}(\ref{lem:tol:i1}), $x_0 \vee x_1 \relgamma y$ and
$x_1 \vee x_2 \relgamma y$.
Since $x_1 \leq x \leq y$, by \ref{MCc}, $a_1 \in \F{x_1}{y}$.
Using \ref{MCc} again, $\P{x_1}{x_0 \vee x_1}\,a_1 \in \F{x_0 \vee x_ 1}{y}$.
Similarly, $\P{x_1}{x_1 \vee x_2}\,a_1 \in \F{x_1 \vee x_ 2}{y}$.
By the premises, $\P{x_0}{x_0 \vee x_1}\,a_0 = \P{x_1}{x_0 \vee x_1}\,a_1
\in \F{x_0 \vee x_ 1}{y}$, so
\begin{alignat*}{2}
\P{x_0}{y}\,a_0 & = \P{x_0 \vee x_1}{y}(\P{x_0}{x_0 \vee x_1}\, a_0) \\
& = \P{x_0 \vee x_1}{y}(\P{x_1}{x_0 \vee x_1}\, a_1) \\
& = \P{x_1}{y}\,a_1 \\
& = \P{x_1 \vee x_2}{y}(\P{x_1}{x_1 \vee x_2}\, a_1) \\
\interject{because of the premise $(x_1,a_1)\sim(x_2,a_2)$,
$\P{x_1}{x_1 \vee x_2}\,a_1 = \P{x_2}{x_1 \vee x_2}\,a_2$, so}
& = \P{x_1 \vee x_2}{y}(\P{x_2}{x_1 \vee x_2}\, a_2) \\
& = \P{x_2}{y}\,a_2
\end{alignat*}
By lem.~\ref{lem:between-equal},
since $x_0, x_2 \leq x_0 \vee x_2 \leq y$,
$\P{x_0}{x_0 \vee x_2}\,a_0 = \P{x_2}{x_0 \vee x_2}\,a_2$
and thus $(x_0, a_0) \sim (x_2, a_2)$.

Reflexivity and symmetry are trivially satisfied.
\end{proof}

\begin{definition} \label{def:sum}
The \emph{sum} $L$ of a \mcs $\scrC$ is defined to be the set $M$
modulo the equivalence relation $\sim$.
\end{definition}

\begin{definition} \label{def:kappa}
The map $\kappa$ is defined to be the canonical projection of $M$ onto $L$
for the equivalence relation $\sim$.
\end{definition}

\begin{definition} \label{def:pi}
For each $x \in S$, we define the canonical mapping
$\pi_x: L_x \rightarrow L: y \mapsto \kappa\,(x, y)$.
\end{definition}

\begin{definition} \label{def:Lam}
We define $\Lambda_x = \pi_x\, L_x$.
For $x \in S$, we define $0_x = \pi_x\,\hatzero_{L_x}$ and
$1_x = \pi_x\,\hatone_{L_x}$.
\end{definition}

\begin{lemma} \label{lem:canon-1-1}
The mappings $\pi_x$ are bijective between $L_x$ and $\Lambda_x$.
\end{lemma}
\begin{proof}
\leavevmode

By the definition of $\Lambda_x$, $\pi_x$ is onto.
To prove that $\pi_x$ is one-to-one:
If $\pi_x(a) = \pi_x(b)$ for $a, b \in L_x$, then
$\kappa\,(x,a) = \kappa\,(x,b)$ and $(x,a) \sim (x,b)$.
Then by def.~\ref{def:sim},
$\P{x}{x}\,a = \P{x}{x}\,b$, which, since $\P{x}{x}$ is the identity
on $L_x$, means $a = b$.
\end{proof}

\begin{lemma}\citeHerr*{(6)}{(6)} \label{lem:(6)-for-Lam}
If $x, y, z \in S$, $x \leq z \leq y$, and $x \relgamma y$
(implying by lem.~\ref{lem:tol}(\ref{lem:tol:i1}),
$x \relgamma z$ and $z \relgamma y$).
Then $\Lambda_x \cap \Lambda_y \subset \Lambda_z$.
\end{lemma}
\begin{proof}
\leavevmode

Given $a \in \Lambda_x \cap \Lambda_y$,
\begin{alignat*}{2}
\kappa\,(x, \pi_x^{-1}\,a) & = a = \kappa\,(y, \pi_y^{-1}\,a) \\
(x, \pi_x^{-1}\,a) & \sim (y, \pi_y^{-1}\,a) \\
\interject{by lem.~\ref{lem:sim-leq-z},}
\pi_y^{-1}\,a & = \P{x}{y}(\pi_x^{-1}\,a) \\
& = \P{z}{y}(\P{x}{z}(\pi_x^{-1}\,a)) \\
\interject{again by lem.~\ref{lem:sim-leq-z},}
(z, \P{x}{z}(\pi_x^{-1}\,a)) & \sim (y, \pi_y^{-1}\,a) \\
\kappa\,(z, \P{x}{z}(\pi_x^{-1}\,a)) & = \kappa\,(y, \pi_y^{-1}\,a) = a
\end{alignat*}
Thus, $a \in \Lambda_z$.
\end{proof}

\begin{lemma}
If $x \leq_\gamma y$ in $S$ and $a \in L_x, \F{x}{y}$, then
$\pi_y(\P{x}{y}\,a) = \pi_x\,a$.
If $x \leq_\gamma y$ in $S$ and $b \in L_y, \I{x}{y}$, then
$\pi_x(\PI{x}{y}\,b) = \pi_y\,b$.
\end{lemma}
\begin{proof}
\leavevmode

\begin{alignat*}{2}
\interject{By \ref{MCd},}
\P{y}{x \vee y}(\P{x}{y}\,a) & = \P{x}{x \vee y}\,a \\
\interject{by def.~\ref{def:sim},}
(y, \P{x}{y}\,a) & \sim (x, a) \\
\interject{by def.~\ref{def:kappa},}
\kappa\,(y, \P{x}{y}\,a) &= \kappa\,(x, a) \\
\interject{by def.~\ref{def:pi},}
\pi_y(\P{x}{y}\,a) & = \pi_x\,a
\end{alignat*}

Dually, we show that $\pi_x(\PI{x}{y}\,b) = \pi_y\,b$.
\end{proof}

\paragraph{The sum is a poset}

\begin{definition} \label{def:leq-sub-x}
For every $v \in S$,
we define on $L$ the binary relation $\leq_v$:
Given $x, y \in L$, $x \leq_v y$ iff
$x, y \in \Lambda_v$ and $\pi_v^{-1}\,x \leq \pi_v^{-1}\,y$
(as elements of $L_v$).
We define $x \geq_v y$ iff $y \leq_v x$.
We define $x <_v y$ iff $\pi_v^{-1}\,x < \pi_v^{-1}\,y$, or
equivalently
$x \leq_v y$ and $\pi_v^{-1}\,x \neq \pi_v^{-1}\,y$, or
equivalently (since $\pi_v$ is bijective)
$x \leq_v y$ and $x \neq y$.
\end{definition}

\begin{lemma} \label{lem:leq-sub-x-po}
For every $x \in S$,
the relation $\leq_x$ is a partial ordering on $\Lambda_x$.
\end{lemma}
\begin{proof}
\leavevmode

This follows from def.~\ref{def:leq-sub-x} and
lem.~\ref{lem:canon-1-1} ($\pi_x^{-1}$ is bijective).
\end{proof}

\begin{lemma} \citeHerr*{(9)}{(9)} \label{lem:(9)}
For $a, b \in L$ and $x, y \in S$, if $a \leq_x b$ and $a \in \Lambda_y$, then
$a \leq_{x \vee y} b$ (which implies $a, b \in \Lambda_{x \vee y}$).
\end{lemma}
\begin{proof}
\leavevmode
\begin{alignat*}{2}
\interject{Since $a \in \Lambda_x$ and $a \in \Lambda_y$,
$\pi_x^{-1}\, a$ and $\pi_y^{-1}\,a$ exist and by
def.~\ref{def:sum}, \ref{def:kappa}, and~\ref{def:pi},}
\kappa\,(x, \pi_x^{-1}\,a) & = a = \kappa(y, \pi_y^{-1}\,a) \\
\interject{by def.~\ref{def:kappa},}
(x, \pi_x^{-1}\,a) & \sim (y, \pi_y^{-1}\,a) \\
\interject{by def.~\ref{def:sim},
(a) $\pi_x^{-1}\,a \in \F{x}{x \vee y}$, $\pi_y^{-1}\,a \in \F{y}{x \vee y}$, and}
\P{x}{x \vee y}(\pi_x^{-1}\,a) & = \P{y}{x \vee y}(\pi_y^{-1}\,a) \\
\interject{since $\P{x \vee y}{x \vee y}$ is the identity,}
\P{x}{x \vee y}(\pi_x^{-1}\,a) & =
\P{x \vee y}{x \vee y}(\P{y}{x \vee y}(\pi_y^{-1}\,a)) \\
\interject{by def.~\ref{def:sim},}
(x, \pi_x^{-1} a) & \sim (x \vee y, \P{y}{x \vee y}(\pi_y^{-1}\,a)) \\
\interject{by def.~\ref{def:kappa},}
\kappa\,(x, \pi_x^{-1}\,a) & = \kappa\,(x \vee y, \P{y}{x \vee y}(\pi_y^{-1}\,a)) \\
\interject{applying def.~\ref{def:pi} to both sides,}
a & = \pi_{x \vee y}(\P{y}{x \vee y}(\pi_y^{-1}\,a)) \\
\pi_{x \vee y}^{-1}\,a & = \P{y}{x \vee y}(\pi_y^{-1}\,a) \tag{b} \\
\interject{And thus, (c) $a \in \Lambda_{x \vee y}$.}
\interject{By hypothesis,}
a & \leq_x b \\
\interject{by def.~\ref{def:leq-sub-x},}
\pi_x^{-1}\,a & \leq \pi_x^{-1}\,b \\
\interject{by (a) above, $\pi_x^{-1}\,a \in \F{x}{x \vee y}$,
so $\pi_x^{-1}\,b \in \F{x}{x \vee y}$, and
since $\P{x}{x \vee y}$ is an isomorphism on $\F{x}{x \vee y}$,}
\P{x}{x \vee y}(\pi_x^{-1}\,a) & \leq \P{x}{x \vee y}(\pi_x^{-1}\,b) \tag{d} \\
\interject{since $\P{x \vee y}{x \vee y}$ is the identity,}
\P{x}{x \vee y}(\pi_x^{-1}\,b) & =
\P{x \vee y}{x \vee y}(\P{x}{x \vee y}(\pi_x^{-1}\,b)) \\
\interject{by def.~\ref{def:sim},}
(x, \pi_x^{-1} b) & \sim (x \vee y, \P{x}{x \vee y}(\pi_x^{-1}\,b)) \\
\interject{by def.~\ref{def:kappa},}
\kappa\,(x, \pi_x^{-1} b) & = \kappa\,(x \vee y, \P{x}{x \vee y}(\pi_x^{-1}\,b)) \\
\interject{applying def.~\ref{def:pi} to both sides,}
b & = \pi_{x \vee y}(\P{x}{x \vee y}(\pi_x^{-1}\,b)) \\
\pi_{x \vee y}^{-1}\,b & = \P{x}{x \vee y}(\pi_x^{-1}\,b) \tag{e} \\
\interject{And thus, (f) $b \in \Lambda_{x \vee y}$.}
\interject{Substituting (b) and (d) into (e),}
\pi_{x \vee y}^{-1}\,a & \leq \pi_{x \vee y}^{-1}\,b \tag{g} \\
\interject{by def.~\ref{def:leq-sub-x}, (c), (f), and (g) together show}
a & \leq_{x \vee y} b
\end{alignat*}
\end{proof}

\begin{lemma} \label{lem:(9)-dual}
For $a, b \in L$ and $x, y \in S$, if $a \leq_y b$ and $b \in \Lambda_x$, then
$a \leq_{x \wedge y} b$ (which implies $a, b \in \Lambda_{x \wedge y}$).
\end{lemma}
\begin{proof}
\leavevmode

This is proved dually to lem.~\ref{lem:(9)}.
\end{proof}

\begin{lemma} \label{lem:(2)}
For $a, b \in L$ and $x \leq y$ in $S$
with $a, b \in \Lambda_x \cap \Lambda_y$,
$a \leq_x b$ iff $a \leq_y b$.
\end{lemma}
\begin{proof}
\leavevmode

Regarding $\Rightarrow$:
By hypothesis $a \leq_x b$ and $a \in \Lambda_y$, so by
lem.~\ref{lem:(9)},
$a \leq_{x \vee y} b$, which is equivalent to $a \leq_y b$.

Regarding $\Leftarrow$:
This is proved dually to the $\Rightarrow$ case.
\end{proof}

\begin{definition} \label{def:asc-seq-new}
We define an \emph{ascending sequence of length $n$} to be a pair
$((x_i)_{1 \leq i \leq n}, (a_i)_{0 \leq i \leq n})$ where:
\begin{enumerate}
\item for all $1 \leq i \leq n$, $x_i \in S$ and
for all $0 \leq i \leq n$, $a_i \in L$,
\item $x_1 \leq x_2 \leq \cdots \leq x_n$, and
\item $a_0 \leq_{x_1} a_1 \leq_{x_2} \cdots \leq_{x_n} a_n$.
\end{enumerate}
We call the $x_\bullet$ \emph{blocks} of the ascending sequence (since
they are elements of $S$ which index the blocks $L_\bullet$), and we call the
$a_\bullet$ the \emph{elements} of the sequence.
By sloppy language,
we say that the ascending sequence is \emph{from $a$ to $b$}.
\end{definition}

Note that by def.~\ref{def:leq-sub-x}, in any ascending sequence
$(x_\bullet, a_\bullet)$ of length $n$,
for all $1 \leq i \leq n$, $a_{i-1}, a_i \in \Lambda_{x_i}$.
Thus, for all $1 \leq i \leq n-1$,
$a_i \in \Lambda_{x_i} \cap \Lambda_{x_{i+1}}$.

(Lemma 4.28 in v1 has been renumbered
lemma~\forwardref{lem:asc-seq-lift-strong}.)

\begin{lemma} \label{lem:asc-seq-prepend}
Given an ascending sequence $(x_\bullet, a_\bullet)$ of length $n$,
$y \in S$ where $y \leq x_1$, and $b \leq_y a_0$,
then the sequence
$(x_\bullet^\prime, a_\bullet^\prime)$ of length $n+1$ defined by
\begin{alignat*}{2}
x_i^\prime & =
\begin{cases}
y & \rmif i = 1 \\
x_{i-1} & \rmif 1 < i \leq n+1
\end{cases} \\
a_i^\prime & =
\begin{cases}
b & \rmif i = 0 \\
a_{i-1} & \rmif 1 \leq i \leq n+1
\end{cases}
\end{alignat*}
is a valid ascending sequence.
\end{lemma}

\begin{subnumbering}

\begin{lemma} \label{lem:asc-seq-append}
Given an ascending sequence $(x_\bullet, a_\bullet)$ of length $n$
$y \in S$ where $y \geq x_n$, and $a_n \leq_y b$,
then the sequence
$(x_\bullet^\prime, a_\bullet^\prime)$ of length $n+1$ defined by
\begin{alignat*}{2}
x_i^\prime & =
\begin{cases}
x_i & \rmif 1 \leq i < n+1 \\
y & \rmif i = n+1
\end{cases} \\
a_i^\prime & =
\begin{cases}
a_i & \rmif 0 \leq i < n+1 \\
b& \rmif i = n+1
\end{cases}
\end{alignat*}
is a valid ascending sequence.
\end{lemma}

\begin{definition} \label{def:asc-seq-prepend-append}
The construction in lem.~\ref{lem:asc-seq-prepend} is called
\emph{prepending an element $b$ and block $y$}.
The construction in lem.~\ref{lem:asc-seq-append} is called
\emph{appending an element $b$ and block $y$}.
\end{definition}

\begin{lemma} \label{lem:asc-seq-lift-weak}
Given an ascending sequence $(x_\bullet, a_\bullet)$ of length $n$
and a $y \in S$ where $a_0 \in \Lambda_y$,
then the sequence
$(x_\bullet^\prime, a_\bullet^\prime)$ of length $n$ defined by
$x_i^\prime = x_i \vee y$ and $a_i^\prime = a_i$
is a valid ascending sequence.
\end{lemma}
\begin{proof} \disconnect
Properties (1) and (2) of an ascending sequence
for $(x_\bullet^\prime, a_\bullet^\prime)$
are trivial; what remains is to show is property (3), that for
all $1 \leq i \leq n+1$, $a_{i-1}^\prime \leq_{x_i^\prime} a_i^\prime$,
which is equivalent to $a_{i-1} \leq_{x_i \vee y} a_i$.

We prove this by induction on $n$.
The base case is $n = 0$, in which case there is one element, $a_0$
and no blocks, and the lemma is trivial.

For $n > 0$,
we apply lem.~\ref{lem:(9)} to $a_0 \leq_{x_1} a_1$ and
$a_0 \in \Lambda_y$ to show $a_0 \leq_{x_1 \vee y} a_1$ and
$a_1 \in \Lambda_{x_1 \vee y}$.
If $n = 1$, this shows that $((x_1 \vee y), (a_0, a_1))$ is an
ascending sequence and proves the lemma.

If $n > 1$,
we apply this lemma inductively to the ascending sequence
$((x_i)_{2 \leq i \leq n}, (a_i)_{2 \leq i \leq n})$ and block $x_1 \vee y$,
which has length $n-1$.
The resulting ascending sequence has first block
$x_2 \vee (x_1 \vee y) = x_2 \vee y$ and first
element $a_2$.  Since $x_1 \vee y \leq x_2 \vee y$ and
$a_0 \leq_{x_1 \vee y} a_1$, we can prepend
$x_2 \vee y$ and $a_0$ to the resulting ascending sequence to
construct the ascending sequence that proves the lemma.
\end{proof}

\begin{lemma} \label{lem:asc-seq-lower}
Given an ascending sequence $(x_\bullet, a_\bullet)$ of length $n$
and a $y \in S$ where $a_n \in \Lambda_y$,
then the sequence
$(x_\bullet^\prime, a_\bullet^\prime)$ of length $n$ defined by
$x_i^\prime = x_i \wedge y$ and $a_i^\prime = a_i$
is a valid ascending sequence.
\end{lemma}
\begin{proof}
We prove this dually to lem.~\ref{lem:asc-seq-lift-weak} using
lem.~\ref{lem:(9)-dual}.
\end{proof}

\begin{lemma} \label{lem:asc-seq-concat}
Given an ascending sequence $(x_\bullet, a_\bullet)$ of length $n$
and an ascending sequence $(y_\bullet, b_\bullet)$ of length $m$
with $a_n = b_0$, then the sequence
$(z_\bullet, d_\bullet$ of length $n+m$ defined by
\begin{alignat*}{2}
z_i & =
\begin{cases}
x_i & \rmif 1 \leq i \leq n \\
y_{i-n} \vee x_n & \rmif n+1 \leq n+m
\end{cases} \\
d_i & =
\begin{cases}
a_i & \rmif 0 \leq i < n \\
a_n & \rmif i = n \\
b_{i-n} & \rmif n < i \leq n+m
\end{cases}
\end{alignat*}
is a valid ascending sequence.
\end{lemma}
\begin{proof}
Applying lem.~\ref{lem:asc-seq-lift-weak} to $(y_\bullet, b_\bullet)$
and $x_n$, construct the ascending
sequence $((y_i \vee x_n)_{1 \leq i \leq m}, (b_i)_{0 \leq i \leq m})$.
The validity of the new sequence, the validity of $(x_\bullet, a_\bullet)$,
and $z_n = x_n \leq y_1 \vee x_n = z_{n+1}$
together prove all the needed facts to show that
$(z_\bullet, d_\bullet)$ is a valid ascending sequence.
\end{proof}

\begin{lemma} \label{lem:asc-seq-lift-strong}
Given an ascending sequence $(x_\bullet, a_\bullet)$ of length $n$
and an $x \in S$ where $a_0 \in \Lambda_x$,
then the sequence $(x_\bullet^\prime, a_\bullet^\prime)$ of length
$n+1$ defined by
\begin{alignat*}{2}
x_i^\prime & =
\begin{cases}
x & \rmif i = 1 \\
x_{i-1} \vee x & \rmif 1 < i \leq n+1
\end{cases} \\
a_i^\prime & =
\begin{cases}
a_0 & \rmif i = 0 \\
a_{i-1} & \rmif 1 \leq i \leq n+1
\end{cases}
\end{alignat*}
is a valid ascending sequence.
\end{lemma}
\begin{proof} \disconnect
Apply lem.~\ref{lem:asc-seq-lift-weak} to $(x_\bullet, a_\bullet)$
and $x$.  The resulting sequence's first block is $x_1 \vee x$.
Then apply lem.~\ref{lem:asc-seq-prepend} to $a \leq_x a$
and $x \leq x_1 \vee x$ to
prepend to the resulting sequence the block $x$ and element $a$.
\end{proof}

\end{subnumbering}

\begin{definition}\citeHerr*{sec.~2}{sec.~2} \label{def:L-leq}
For $a, b \in L$, we define $a \leq b$ iff there exists
an ascending sequence $(x_\bullet, a_\bullet)$ of some length $n$ with
$a = a_0$ and $b = a_n$.
\end{definition}

\begin{lemma}\citeHerr*{(10)}{(10)} \label{lem:(10)+}
If $a \leq b$ in $L$ and if $a \in \Lambda_x$ for a specified
$x \in S$,
there is an ascending sequence $(y_\bullet, c_\bullet)$ of some length $n$
with $y_1 = x$, $c_0 = a$, and $c_n = b$ (which also shows $a \leq b$).
\end{lemma}
\begin{proof} \disconnect
By def.~\ref{def:L-leq}, there exists an ascending sequence
from $a$ to $b$.  Define $y$ to be its first block, so
$a \in \Lambda_y$.
Applying lem.~\ref{lem:asc-seq-lift-weak} to the ascending sequence
and $x$ produces an ascending sequence from $a$ to $b$
with first block $y \vee x$ and showing $a \in \Lambda_{y \vee x}$.
Knowing that $a \leq_x a$ and $x \leq y \vee x$,
apply lem.~\ref{lem:asc-seq-prepend} to the ascending sequence
to prepend block $x$ and element $a$.
The resulting ascending sequence satisfies the requirements of the
conclusion.
\end{proof}

\begin{lemma} (part of \citeHerr*{(14)}{(14)}) \label{lem:(10)+strict}
If $a \leq b$ in $L$ and if $a \in \Lambda_x$ for a specified
$x \in S$,
there is an ascending sequence $(y_\bullet, c_\bullet)$ of some length $n$
with $y_1 \geq x$, $c_0 = a$, $c_n = b$, and
the elements $y_\bullet$ are strictly increasing:
$y_i <_{c_{i+1}} y_{i+1}$.
\end{lemma}
\begin{proof} \disconnect
By lem.~\ref{lem:(10)+}, there exists an ascending sequence
$(y_\bullet, c_\bullet)$ of length $n$ that satisfies all of the
requirements except
that there may be adjacent pairs of elements that are are not strictly
increasing, that is, for some $1 \leq i \leq n$,
$y_{i-1} = y_i$.

We can meet that requirement by repeatedly deleting the second element
of any adjacent pair of equal elements.  Since the sequence has a
finite length, this process eventually terminates.
(Note that if $a = b$, the
sequence may be reduced to length 0, which is valid and satisfies the
requirements.)

We remove adjacent equal elements as follows:
Let $y_i {\not<}_{c_{i+1}} y_{i+1}$ for some $0 \leq i < n$.
Since $y_i \leq_{c_{i+1}} y_{i+1}$; $y_i$ and $y_{i+1}$ must have the same images
under $\pi_{c_{i+1}}^{-1}$, and since $\pi_{c_{i+1}}^{-1}$ is bijective, they
must be equal as elements of $L$.  Then
\begin{gather*}
y_0 \leq_{c_1} \cdots \leq_{c_i} y_i = y_{i+1} \leq_{c_{i+2}} y_{i+2}
\leq_{c_{i+3}} \cdots \leq_{c_n} y_n \\
c_1 \leq \cdots \leq c_i \leq c_{i+1} \leq c_{i+2} \leq \cdots \leq c_n
\end{gather*}
We create a new ascending sequence:
\begin{gather*}
y_0 \leq_{c_1} \cdots \leq_{c_i} y_i \leq_{c_{i+2}} y_{i+2} \leq_{c_{i+3}}
\cdots \leq_{c_n} y_n \\
c_1 \leq \cdots \leq c_i \leq c_{i+2} \leq \cdots \leq c_n
\end{gather*}
This a valid ascending sequence because the only new condition on its
validity is $y_i \leq_{c_{i+2}}$, which follows because $y_i = y_{i+1}$.

Note that if $i = 0$, then the first element of the sequence remains
unchanged as $a$, but the first block is removed and replaced with
$y_1$, which is $\geq x$ but may be $> x$.
If $i = n-1$, the last element of the sequence is replaced with $y_{n-1}$
which is equal to the former first element $y_n$ and the last block is
removed and replaced with $c_{n-1}$.
\end{proof}

\begin{lemma} \label{lem:ref}
The relation $\leq$ on $L$ is reflexive.
\end{lemma}
\begin{proof}
\leavevmode

The ascending sequence $((), (a))$ (of length 1)
shows $a \leq a$.
\end{proof}

\begin{lemma} \label{lem:antisymmetric}
The relation $\leq$ on $L$ is antisymmetric.
\end{lemma}
\begin{proof} \disconnect
Assume $a, b \in L$ with $a \leq b$ and $b \leq a$.
By def.~\ref{def:L-leq} there is an ascending sequence of length $n$
from $a$ to $b$ and an ascending sequence of length $m$ from $b$ to
$a$.
Using lem.~\ref{lem:asc-seq-concat}, we can
concatenate these ascending sequences to
construct an ascending sequence $(x_\bullet, c_\bullet)$
of length $n+m$ with $c_0 = a$, $c_n = b$, $c_{n+m} = a$,
and $c_0 = a = c_{n+m} \in \Lambda_{x_{n+m}}$.
Applying lem.~\ref{lem:asc-seq-lift-weak} to $(x_\bullet, c_\bullet)$
and $x_{n+m}$,
$((x_i \vee x_{n+m})_{1 \leq i \leq n+m}, (c_i)_{1 \leq i \leq x+m})$
is an ascending sequence.
But since $x_i \leq x_{n+m}$ for all $i$, this ascending sequence is
equal to $((x_{n+m})_{1 \leq i \leq n+m}, (c_i)_{1 \leq i \leq x+m})$.
By the transitivity of $\leq_{x_{n+m}}$, this shows that
$a \leq_{x_{n+m}} b \leq_{x_{n+m}} a$, which implies that
$\pi_{x_{n+m}}^{-1}\,a = \pi_{x_{n+m}}^{-1}\,b$, and since
$\pi_{x_{n+m}}$ is bijective, $a=$b.
\end{proof}

\begin{lemma} \label{lem:trans}
The relation $\leq$ on $L$ is transitive.
\end{lemma}
\begin{proof} \disconnect
Assume $a, b, c \in L$ with $a \leq b$ and $b \leq c$.
By def.~\ref{def:L-leq}, there are ascending sequences
from $a$ to $b$ and from $b$ to $c$.
Using lem.~\ref{lem:asc-seq-concat}, these can be concatenated to form
an ascending sequence from $a$ to $c$, which shows that $a \leq c$.
\end{proof}

\begin{lemma} \label{lem:poset}
$\leq$ on $L$ is a partial ordering.
\end{lemma}
\begin{proof}
\leavevmode

This is proved by
lem.~\ref{lem:ref}, \ref{lem:antisymmetric}, and~\ref{lem:trans}.
\end{proof}

\begin{lemma}\citeHerr*{(11)}{(11)} \label{lem:(11)}
If $a \leq c \leq b$ in $L$ and for some $x \in S$,
$a, b \in \Lambda_x$, then $c \in \Lambda_x$ and
$a \leq_x c \leq_x b$.
\end{lemma}
\begin{proof}
\leavevmode

As in previous proofs, by def.~\ref{def:L-leq}, there exist ascending
sequences from $a$ to
$c$ and from $c$ to $b$.  Using lem.~\ref{lem:(10)+}, these sequences can be
concatenated to produce
one sequence with starting element $a$, ending element $b$, and with $c$ as an
element.  We apply lem.~\ref{lem:(10)+} to the concatenated sequence
and $x$ to produce such a sequence with starting block $x$.

Since the starting block of the final sequence is $x$,
every block is $\geq x$.  The ending element is $b$ and
$b \in \Lambda_x$.  Applying lem.~\ref{lem:(9)-dual} inductively from
the end of the sequence to the beginning proves that every element
$\in \Lambda_x$ and every consecutive pair of elements is $\leq_x$.
That shows that $c \in \Lambda_x$ and by
lem.~\ref{lem:leq-sub-x-po}, $a \leq_x c \leq_x b$.
\end{proof}

\begin{lemma} \label{lem:pi-poset-iso}
If $a, b \in L_x$, then $a \leq b$ in $L_x$ iff
$\pi_x\,a \leq \pi_x\,b$ in $L$.
Equivalently, if $a^\prime, b^\prime \in \Lambda_x$,
then $a^\prime \leq_x b^\prime$ iff $a^\prime \leq b^\prime$.
\end{lemma}
\begin{proof}
\leavevmode

Regarding $\Rightarrow$:
If $a \leq b$ in $L_x$, then $((x), (\pi_x(a), \pi_x(b)))$
is an ascending sequence
(of length 1) which shows that $\pi_x\,a \leq \pi_x\,b$.

Regarding $\Leftarrow$:
If $\pi_x\,a \leq \pi_x\,b$, since $\pi_x(a), \pi_x(b) \in \Lambda_x$,
by lem.~\ref{lem:(11)},
$\pi_x\,a \leq_x \pi_x\,b$, which is equivalent to $a \leq b$ in $L_x$.
\end{proof}

\paragraph{\texorpdfstring{The $\Pi$ intervals of $S$}%
{The \83\240\ intervals of S}}

\begin{definition} \label{def:Pi}
For any $a \in L$, we define
$\Pi_a = \{ x \in S \mvert a \in \Lambda_x \}$,
the set of $x \in S$ for which $L_x$ has an element
that maps to $a$ via $\pi_x$.
\end{definition}

\begin{lemma}\flagstart$\dagger$ \label{lem:0-1-extreme}
For any $x \in S$, $x$ is the maximum element of $\Pi_{0_x}$
and the minimum element of $\Pi_{1_x}$.
\end{lemma}
\begin{proof}
\leavevmode

Suppose $z \in \Pi_{0_x}$ and so $0_x \in \Lambda_z$.
Then there exists $a \in L_z$ such that
$(z, a) \sim (x, \hatzero_{L_x})$, requiring
$\P{z}{x \vee z}\,a = \P{x}{x \vee z}\,\hatzero_{L_x}$.
Thus $\hatzero_x \in \F{x}{x \vee z}$.
But since $\hatzero_{L_x}$ is the minimum of $L_x$, by \ref{MCb2},
it can only be in $\F{x}{x \vee z}$ if $x \vee z = x$.
Thus $z \leq x$, showing $x$ is the maximum of $\Pi_{0_x}$.

Dually, we prove $x$ is the minimum element of $\Pi_{1_x}$.
\end{proof}

\begin{lemma} \label{lem:Pi-conv-sub}
For any $a \in L$, $\Pi_a$ is a convex sublattice of $S$.
\end{lemma}
\begin{proof}
\leavevmode

Given $x, x^\prime \in \Pi_a$, there exist $b \in L_x$ and
$b^\prime \in L_{x^\prime}$ such that $a = \kappa\,(x,b)
= \kappa\,(x^\prime,b^\prime)$. By def.~\ref{def:sim}, that implies
$(x,b) \sim (x^\prime,b^\prime)$ and so
$x \relgamma x^\prime$, $b \in \F{x}{x \vee x^\prime}$,
$b^\prime \in \F{x^\prime}{x \vee x^\prime}$, and
$\P{x}{x \vee x^\prime}\,b = \P{x^\prime}{x \vee x^\prime}\,b^\prime$.
Defining $x^{\prime\prime} = x \vee x^\prime$ and
$b^{\prime\prime} = \P{x}{x \vee x^\prime}$, we can assemble that
$x \relgamma x^{\prime\prime}$ and
$(x,b) \sim (x^{\prime\prime},b^{\prime\prime})$, and so
$\kappa\,(x^{\prime\prime},b^{\prime\prime}) = a$, showing
$x^{\prime\prime} \in \Pi_a$.  This shows that $\Pi_a$
is closed under joins.

Dually, using lem.~\ref{lem:sim-dual} instead of def.~\ref{def:sim},
we prove that $\Pi_a$ is closed under meets.

Given $x, x^{\prime\prime} \in \Pi_a$ and
$x \leq x^\prime \leq x^{\prime\prime}$,
there exist $b \in L_x$ and
$b^{\prime\prime} \in L_{x^{\prime\prime}}$ such that $a = \kappa\,(x,b)
= \kappa\,(x^{\prime\prime},b)$.
By def.~\ref{def:sim}, $x \relgamma x^{\prime\prime}$, so by
lem.~\ref{lem:tol}(\ref{lem:tol:i1}) $x \relgamma x^\prime$ and
by lem.~\ref{lem:(6)}, $b \in \F{x}{x^{\prime\prime}}$, and so
$b \in \F{x}{x^\prime}$.
Defining $b^\prime = \P{x}{x^{\prime\prime}}\, b$,
$\kappa\,(x^\prime,b^\prime) = a$ and so $x^\prime \in \Pi_a$.
Thus $\Pi_a$ is convex.
\end{proof}

\begin{subnumbering}

\begin{lemma}\flagstart$\dagger$ \label{lem:Pi-interval}
For any $a \in L$, $\Pi_a$ is a finite interval in $S$.
\end{lemma}
\begin{proof} \disconnect
By lem.~\ref{lem:Pi-conv-sub}, $\Pi_a$ is a convex sublattice of $L$.
Assume that there is no maximum element of $\Pi_a$.  Then there is
an infinite ascending chain in $\Pi_a$,
$x_0 < x_1 < x_x < \cdots$.
For each $y \in \Pi_a$, define $\tau_y$ to be
$\rho_{L_y}(\pi^{-1}_y\,a)$, the rank of
$\pi^{-1}_y\,a$ in $L_y$, which since $L_y$ is a finite modular
lattice, is well-defined and $\in \mbbN$.

$\tau$ is a strictly decreasing map: Let $x < y$ in $\Pi_a$.
Since every $\PI{z}{w}$ is an isomorphism on its domain and image,
which are convex in $L_z$ and $L_w$, respectively,
$\PI{z}{w}$ preserves relative ranks.
\begin{alignat*}{2}
\tau_y
& = \rho_{L_y}(\pi^{-1}_y\,a) \\
& = \rho_{L_y}(\pi^{-1}_y\,a) - \rho_{L_y}(\hatzero_{L_y}) \\
& = \rho_{L_x}(\PI{x}{y}(\pi^{-1}_y\,a)) -
\rho_{L_x}(\PI{x}{y}(\hatzero_{L_y})) \\
& = \rho_{L_x}(\pi^{-1}_x\,a) -
\rho_{L_x}(\PI{x}{y}(\hatzero_{L_y})) \\
& = \tau_x - \rho_{L_x}(\PI{x}{y}(\hatzero_{L_y}))
\end{alignat*}
By lem.~\ref{lem:monotone-0-1} we know that
$\PI{x}{y}(\hatzero_{L_y}) = \hatzero^y_x > \hatzero_{L_x}$,
so $\rho_{L_x}(\PI{x}{y}(\hatzero_{L_y})) > 0$,
thus showing $\tau_y < \tau_x$.

This implies that $\tau_{x_0} > \tau_{x_1} > \cdots$, which is
impossible as they are all $\in \mbbN$, and so there is no infinite
ascending chain in
$\Pi_a$.  Since we have shown that $\Pi_a$ is closed under
joins, it then must have a maximum element.

Dually, we show that $\Pi_a$ has minimum element.

Together, these facts show that $\Pi_a$ is an interval in $S$, and
because $S$ is locally finite (by \ref{MCf}), $\Pi_a$ is finite.
\end{proof}

\end{subnumbering}

\paragraph{The relations of the \texorpdfstring{$\Lambda$ intervals of $L$}%
{\83\233\ intervals of L}}

\begin{lemma} \label{lem:MCd-for-Lam}
For any $x, y \in S$, $\Lambda_x \cap \Lambda_y = \Lambda_{x \wedge y} \cap \Lambda_{x \vee y}$.
\end{lemma}
\begin{proof}
\leavevmode

Regarding $\subset$:
Assume $c \in \Lambda_x \cap \Lambda_y$.
Then there exists $a \in L_x$ such that $\pi_x\,a = c$ and
there exists $b \in L_y$ such that $\pi_y\,b = c$.
This implies $(x,a) \sim (y, b)$,
$\kappa\,(x,a) = \kappa\,(y,b) = c$,
$\P{x}{x \vee y}\,a = \P{y}{x \vee y}\,b$, and
$\pi_{x \vee y}(\P{x}{x \vee y}\,a) = \pi_{x \vee y}(\P{y}{x \vee y}\,b) = c$.
Thus, $c \in \Lambda_{x \vee y}$.
Dually, we prove $c \in \Lambda_{x \wedge y}$.

Regarding $\supset$:
Assume $f \in \Lambda_{x \wedge y} \cap \Lambda_{x \vee y}$.
Then there exists $d \in L_{x \wedge y}$ such that $\pi_{x \wedge y}\,d = f$ and
there exists $e \in L_{x \vee y}$ such that $\pi_{x \vee y}\,e = f$.
This implies $(x \wedge y, d) \sim (x \vee y, e)$,
$\P{x \wedge y}{x \vee y}\,d = e$, and so
$\P{x}{x \vee y}(\P{x \wedge y}{x}\,d) = e$.
Thus, $\kappa\,(x, \P{x \wedge y}{x}\,d) = f$, showing $f \in \Lambda_x$.
Similarly, we prove $f \in \Lambda_y$.
\end{proof}

\begin{lemma} \label{lem:Lam-overlap-gamma}
For $x, y \in S$, then $\Lambda_x \cap \Lambda_y \neq \zeroslash$ iff
$x \relgamma y$.
\end{lemma}
\begin{proof}
\leavevmode

Regarding $\Rightarrow$:
Let $e \in \Lambda_x \cap \Lambda_y$.  Then there exists $a \in L_x$ such that
$\pi_x\,a = e$ and $b \in L_y$ such that $\pi_y\,b = e$.
Thus $(x, a) \sim (y, b)$, and by def.~\ref{def:sim},
$x \relgamma y$.

Regarding $\Leftarrow$:
Since $x \relgamma y$, by lem.~\ref{lem:(7)}(1),
$\I{x}{x \vee y} \cap \I{y}{x \vee y} = \I{x \wedge y}{x \vee y}
\neq \zeroslash$.
Let $a \in \I{x \wedge y}{x \vee y}$, so that
$a \in \I{x}{x \vee y}$ and $a \in \I{y}{x \vee y}$, and define
$a_x = \PI{x}{x \vee y}\,a$ and $a_y = \PI{y}{x \vee y}\,a$.
Then $(x, a_x) \sim (y, a_y)$ so that $\pi_x\,a_x = \pi_y\,a_y$.
Since $\pi_x\,a_x \in \Lambda_x$ and $\pi_y\,a_y \in \Lambda_y$,
$\Lambda_x \cap \Lambda_y \neq \zeroslash$.
\end{proof}

\begin{lemma} \citeHerr*{(1)}{(1)} \label{lem:(1)-for-Lam}
If $x \leq_\gamma y$, then
\begin{enumerate}
\item $\Lambda_x \cap \Lambda_y = \pi_x\,\F{x}{y} = \pi_y\,\I{x}{y}$ and
\item $\Lambda_x \cap \Lambda_y$ is a filter of $\Lambda_x$ and an ideal of $\Lambda_y$.
\end{enumerate}
\end{lemma}
\begin{proof}
\leavevmode

Regarding (1) $\subset$:
Assume $e \in \Lambda_x \cap \Lambda_y$.
Then there exists $a \in L_x$ such that $\pi_x\,a = e$ and
$b \in L_y$ such that $\pi_y\,b = e$.
Thus $(x, a) \sim (y, b)$, and by lem.~\ref{lem:sim-leq-z},
$\P{x}{y}\,a = b$, $\PI{x}{y}\,b = a$,
$a \in \F{x}{y}$, and $b \in \I{x}{y}$, showing
$e \in \pi_x\,\F{x}{y}$ and $e \in \pi_y\,\I{x}{y}$.

Regarding (1) $\supset$:
Assume $e \in \pi_x\,\F{x}{y}$.
Since $\F{x}{y} \subset L_x$, $e \in \Lambda_x$.
There exists $a \in \F{x}{y}$ such that $\pi_x\,a = e$.
Define $b = \P{x}{y}\,a$, so $b \in L_y$.
By lem.~\ref{lem:sim-leq-z}, $(x,a) \sim (y, b)$, so $\pi_y\,b = e$ and
$e \in \Lambda_y$.
Dually, we prove that $e \in \pi_y\,\I{x}{y}$ implies $e \in \Lambda_x$.

Regarding (2):
This statement follows trivially from (1) because the $\pi_\bullet$ are poset
isomorphisms by lem.~\ref{lem:pi-poset-iso}.
\end{proof}

\begin{lemma} (part of \citeHerr*{Satz~2.1}{Th.~2.1}) \label{lem:Lambda-inter}
$\Lambda_x$ is an interval in $L$.
\end{lemma}
\begin{proof}
\leavevmode

First, we show that if $a, b \in \Lambda_x$, then there is an upper
bound $c \geq a, b$ in $\Lambda_x$.
Define $c = \pi_x(\pi_x^{-1}\,a \vee \pi_x^{-1}\,b)$.
Clearly $c \in \Lambda_x$.
By def.~\ref{def:leq-sub-x}, $c \geq_x a, b$ and
by lem.~\ref{lem:pi-poset-iso}, $c \geq a, b$.

Dually, we show that if $a, b \in \Lambda_x$, then there is a lower
bound $c \leq a, b$ in $\Lambda_x$.

Since $\Lambda_x$ (the image of $L_x$ under $\pi_x$) is finite and has
upper and lower bounds, it must have maximum and minimum elements.
Because lem.~\ref{lem:(11)} shows that $\Lambda_x$ is convex in $L$,
$\Lambda_x$ is an interval in $L$.
\end{proof}

Note that the following lemma shows that for any $x$,
$\Lambda_x$ (under the partial ordering inherited from $L$) is a lattice,
but not that $L$ as a whole is a lattice (under its partial ordering).

\begin{lemma} (part of \citeHerr*{Satz~2.1}{Th.~2.1}) \label{lem:pi-lattice-iso}
$\pi_x$ is a lattice isomorphism from $L_x$ to $\Lambda_x$.
\end{lemma}
\begin{proof}
\leavevmode

Because $L_x$ is a lattice (\ref{MCg}),
$\pi_x$ is a partial order isomorphism from $L_x$ to $\Lambda_x$
(lem.~\ref{lem:pi-poset-iso}), and
$\Lambda_x$ is an interval in $L$ (lem.~\ref{lem:Lambda-inter}),
$\Lambda_x$ is a lattice (under the partial order of $L$)
and $\pi_x$ is a lattice isomorphism from $L_x$ to $\Lambda_x$.
\end{proof}

\paragraph{The sum is locally finite and has finite covers}

\begin{theorem} \label{th:locally-finite}
$L$ is locally finite.
\end{theorem}
\begin{proof} \disconnect
Choose any $a \leq b$ in $L$.  We must show that the interval $[a, b]$
in $L$ has a finite number of members.
Assume there is an infinite number of distinct $c_i, c_2, c_3, \ldots$
in $L$ with $a \leq c_i \leq b$.

For each $c_i$, by def.~\ref{def:L-leq} and
lem.~\ref{lem:asc-seq-concat}, we can construct ascending sequences
from $a$ to $c_i$ and from $c_i$ to $b$ and then concatenate them to
create an ascending sequence
$(x_{i\bullet}, d_{i\bullet})$ of length $n_i$ with
$d_{i0} = a$, $d_{in_i} = b$, and for some $j_i$, $d_{ij_i} = c_i$.
Thus, for every $i$, $a \in \Lambda_{x_{i1}}$ and
$b \in \Lambda_{x_{in_i}}$.

For every $i$, define $y_i = \bigwedge_{1 \leq j \leq i} x_{j1}$
and $z_i = \bigvee_{1 \leq j \leq i} x_{jn_j}$.
So the $y_\bullet$ are a weakly decreasing sequence in $S$,
the $z_\bullet$ are a weakly increasing sequence in $S$,
$y_1 = x_{11}$, $z_1 = x_{1n_1}$, and
for every $i$,
$y_i \leq x_{i1} \leq x_{in_i} \leq z_i$,
$y_{i+1} = y_i \wedge x_{(i+1)1}$, and
$z_{i+1} = z_i \vee x_{(i+1)n_{i+1}}$.
By lem.~\ref{lem:Pi-conv-sub}, for every $i$, $a \in \Lambda_{y_i}$ and
$b \in \Lambda_{z_i}$.

Let $\rho_i$ be the corank of $\pi_{y_i}^{-1}\,a$ in $L_{y_i}$, that
is the relative rank from $\pi_{y_i}^{-1}\,a$  up to
$\hatone_{L_{y_i}}$.
Let $\sigma_i$ be the rank of $\pi_{z_i}^{-1}\,b$ in $L_{z_i}$, that
is the relative rank from $\hatzero_{L_{z_i}}$ up to
$\pi_{z_i}^{-1}\,b$.
Because each $\PI{y_{i+1}}{y_i}$ is an isomorphism from an ideal of
$L_{y_i}$ to a
filter of $L_{y_{i+1}}$, $\rho_i$ is a decreasing
function of $i$.  Since all $\rho_\bullet$ are non-negative integers,
there is some $k$ for which $\rho_j$ is constant for all $j \geq k$.
Dually, because each $\P{z_i}{z_{i+1}}$ is an isomorphism from a filter of
$L_{z_i}$ to an
ideal of $L_{z_{i+1}}$, $\sigma_i$ is a decreasing
function of $i$.  Since all $\sigma_\bullet$ are non-negative integers,
there is some $p$ for which $\sigma_j$ is constant for all $j \geq p$.
Define $q$ to be the maximum of $k$ and $p$.

For any $i \leq q$, define $w_i = x_{ij_i}$.  By the ascending
sequence $(x_{i\bullet}, d_{i\bullet})$,
$c_i = d_{ij_i} \in \Lambda_{x_{ij_i}} = \Lambda_{w_i}$.
And
\begin{equation*}
y_q \leq y_i \leq x_{i1} \leq x_{x_{ij_i}} = w_i \leq
x_{in_i} \leq z_i \leq z_q.
\end{equation*}

For any $i > q$, apply lem.~\ref{lem:asc-seq-lift-weak}
to $(x_{i\bullet}, d_{i\bullet})$, $y_q$, and
$d_{i0} = a \in \Lambda_{y_q}$
to construct the ascending sequence
$((x_{ij} \vee y_q)_{1 \leq j \leq n_i}, (d_{ij})_{0 \leq j \leq n_i})$.
Note $y_q \leq x_{i1} \vee y_q$ and
since $z_i \geq x_{in_i}$ and $z_i \geq z_q \geq y_q$,
we know $z_i \geq x_{in_i} \vee y_q$.
Now apply lem.~\ref{lem:asc-seq-lower} to
$((x_{ij} \vee y_q)_{1 \leq j \leq n_i},
(d_{ij})_{0 \leq j \leq n_i})$
to construct the ascending sequence
$(((x_{ij} \vee y_q) \wedge z_q)_{1 \leq j \leq n_i},
(d_{ij})_{0 \leq j \leq n_i})$.
Note $z_q \geq (x_{in_i} \vee y_q) \wedge z_q$ and
since $y_q \leq x_{i1} \vee y_q$ and $y_q \leq z_q$,
$y_q \leq (x_{i1} \vee y_q) \wedge z_q$.

For such an $i > q$, define
$w_i = (x_{ij_i} \vee y_q) \wedge z_q$.
By the ascending sequence we have constructed,
$c_i = d_{ij_i} \in \Lambda_{(x_{ij_i} \vee y_q) \wedge z_q} = \Lambda_{w_i}$.
And
\begin{equation*}
y_q \leq (x_{i1} \vee y_q) \wedge z_q \leq
(x_{ij_i} \vee y_q) \wedge z_q = w_i
\leq z_q.
\end{equation*}

Thus for any $i$, $c_i \in \Lambda_{w_i}$ and $y_q \leq w_i \leq z_q$.
Clearly, since each $c_i$ is distinct each has a distinct pair
$(w_i, c_i)$.  But there are only a finite number of such pairs;
since $S$ is locally finite, there are only a finite number of
possible $w_\bullet \in [y_q, z_q]$ and since each
$\Lambda_{w_\bullet}$ is finite, for each possible $w_\bullet$ there
are only a finite number of $c_\bullet$.
This contradicts the assumption that there are an infinite number of
distinct $c_i$ and shows that $L$ is locally finite.
\end{proof}

\begin{remark} \label{rem:locally-finite-monotony}
\textup{
In v1 of this paper, the proof that $L$ is locally finite depended
on monotony, property \ref{MCb2}.
In v2, the proof was sharpened to no longer depend on \ref{MCb2}.
}
\end{remark}

\begin{lemma} (part of \citeHerr*{Satz~2.1}{Th.~2.1}) \label{lem:covers-xfer}
If $a \lessdot b$ in $L$, then there exists $x \in S$ for which
$a, b \in \Lambda_x$ and
$\pi_x^{-1}\,a \lessdot \pi_x^{-1}\,b$.
If $a \lessdot b$ in $L_x$ for some $x \in S$, then
$\pi_x\,a \lessdot \pi_x\,b$.
\end{lemma}
\begin{proof}
\leavevmode

Regarding the first statement:

If $a \lessdot b$ in $L$ then there exists an ascending sequence
$(x_\bullet,c_\bullet)$ of length $n$ with the first element $= a$ and the last
element $= b$.
All of the elements in the sequence must be either $a$ or $b$,
because if there was an element $c$ in the sequence that is $\neq a, b$,
the sequence could be cut into two parts, with the first one showing
that $a \leq c$ and the second one showing that $c \leq b$, which
would contradict that $a \lessdot b$.

Because the first element of the sequence is $a$ and the last element is $b$,
there is some $0 \leq j < n$ for which
$c_j = a$ and $c_{j+1} = b$, and so $a \leq_{x_{j+1}} b$.
Define $x = x_{j+1}$.  By def.~\ref{def:leq-sub-x} and the fact that
$a \neq b$,
$\pi_x^{-1}\,a < \pi_x^{-1}\,b$.
By lem.~\ref{lem:pi-poset-iso} and ~\ref{lem:Lambda-inter},
if there was some $c \in L_x$ such that
$\pi_x^{-1}\,a < c < \pi_x^{-1}\,b$, we would have $a < \pi_x\,c < b$, which
would contradict that $a \lessdot b$.
Thus $\pi_x^{-1}\,a \lessdot \pi_x^{-1}\,b$.

Regarding the second statement:
If $a \lessdot b$ in $L_x$ for some $x \in S$ then
because $\pi_x$ is an isomorphism into an interval of $L$,
$\pi_x\,a \lessdot \pi_x\,b$.
\end{proof}

\begin{theorem}\flagstart$\dagger$ \label{th:finite-cover}
$L$ has finite covers.
\end{theorem}
\begin{proof}
\leavevmode

Select an $x \in L$.
For any $y$ for which $x \lessdot y$,
by lem.~\ref{lem:covers-xfer}, there is a $z \in S$ for which
$\pi_z^{-1}\,x \lessdot \pi_z^{-1}\,y$.
Thus, for any such $y$ there is one or more combinations that satisfy:
\begin{enumerate}
\item $z \in \Pi_x$ and
\item $a \in L_z$,
\end{enumerate}
and for any such combination, $y = \pi_x\,a$.
But by lem.~\ref{lem:Pi-interval}, for any fixed $x$,
the number of such $z$'s is
finite, and because each $L_z$ is finite,
the number of such $a$'s for any particular $z$ is
finite.  Thus there are a finite number of combinations that satisfy
the constraints, and so there are a finite of upper covers of $x$.

We prove there are a finite number of lower covers of $x$ dually.
\end{proof}

\begin{remark} \label{rem:finite-cover-monotony}
\textup{
Monotony (\ref{MCb2}) is required to prove $L$ has finite covers:
Let $S$ be $\mbbN$ ordered by $\leq$, and
$L_i$ be $B_i$, the Boolean algebra with $i$ atoms.
Let $\P{i}{j}$ be the injection from $B_i$ into $B_j$
generated by mapping the $i$ atoms of
$B_i$ to the first $i$ atoms of $B_j$.
Then the $\hatzero$ of sum $L$ is covered by an infinite number of atoms.
}
\end{remark}

\paragraph{The sum is a lattice}


\begin{lemma}(part of \citeHerr*{(12)}{(12)}) \label{lem:(12)-polytone}
For all $x \leq y$ in $S$, $0_x \leq 0_y$ and $1_x \leq 1_y$.
\end{lemma}
\begin{proof} \disconnect
The statement is trivial if $x = y$.
If $x \lessdot y$, then by \ref{MCb1}, $x \relgamma y$,
and necessarily $\hatzero_{L_x} \leq \Z{x}{y}$.
Since $\pi_x$ is an isomorphism,
$0_x = \pi_x\,\hatzero_{L_x} \leq \pi_x\,\Z{x}{y} = \pi_y\,\hatzero_{L_y} = 0_y$.
For $x < y$ but not $x \lessdot y$,
induction on a saturated chain from $x$ to $y$ shows $0_x \leq 0_y$.

Dually, we prove $x \leq y$ implies $1_x \leq 1_y$,
\end{proof}

\begin{subnumbering}

\begin{lemma}\flagstart$\dagger$(part of \citeHerr*{(12)}{(12)}) \label{lem:(12)-monotone}
For all $x, y \in S$:
\begin{enumerate}
\item $x < y$ iff $0_x < 0_y$ iff $1_x < 1_y$,
\item $x \leq y$ iff $0_x \leq 0_y$ iff $1_x \leq 1_y$, and
\item $x \incomp y$ iff $0_x \incomp 0_y$ iff $1_x \incomp 1_y$.%
\footnote{We use $x \incomp y$ to denote that $x$ and $y$ are
incomparable, that is, $x \not\leq y$ and $y \not\leq x$.}
\end{enumerate}
\end{lemma}
\begin{proof}
\leavevmode

Regarding $x < y$ implies $0_x < 0_y$:
Assume $x \lessdot y$.  Then by \ref{MCb1}, $x \relgamma y$,
and by lem.~\ref{lem:monotone-0-1},
$\hatzero_{L_x} < \Z{x}{y}$.
Since $\pi_x$ is an isomorphism,
$0_x = \pi_x\,\hatzero_{L_x} < \pi_x\,\Z{x}{y} = \pi_y\,\hatzero_{L_y} = 0_y$.
For general $x < y$,
induction on a saturated chain from $x$ to $y$ shows $0_x < 0_y$.

Regarding $x \leq y$ implies $0_x \leq 0_y$:
This follows directly from the preceding statement.

Regarding $0_x \leq 0_y$ implies $x \leq y$:
By def.~\ref{def:L-leq} and lem.~\ref{lem:(10)+},
we can construct an ascending sequence whose first element is $0_x$,
last element is $0_y$, and first block is $x$.
Define $z$ to be the last block.  Necessarily $x \leq z$
and $0_y \in \Lambda_z$.
Thus, $z \in \Pi_{0_y}$, and by lem.~\ref{lem:0-1-extreme},
$z \leq y$, which implies $x \leq y$.

Regarding $0_x < 0_y$ implies $x < y$:
This follows directly from the preceding statement.

Regarding $x \incomp y$ iff $0_x \incomp 0_y$:
This follows directly from the preceding statements.

Dually, we prove:
$x < y$ implies $1_x < 1_y$,
$x \leq y$ implies $1_x \leq 1_y$,
$1_x \leq 1_y$ implies $x \leq y$,
$1_x < 1_y$ implies $x < y$, and
$x \incomp y$ iff $1_x \incomp 1_y$.
\end{proof}

\end{subnumbering}

\begin{remark}\flagstart$\dagger$ \label{rem:incomp}
\textup{
Note that if $x \incomp y$ but $x \relgamma y$, then
$0_x \incomp 0_y$ and $1_x \incomp 1_y$ by
lem.~\ref{lem:(12)-monotone}, but since
lem.~\ref{lem:Lam-overlap-gamma} ensures there exists an
$a \in \Lambda_x \cap \Lambda_y$,
$0_x, 0_y \leq a \leq 1_x, 1_y$.
}
\end{remark}

\begin{lemma}\flagstart$\dagger$ \label{lem:Lam-incomp}
If $x \neq y$ in $S$, then $\Lambda_x \not\subset \Lambda_y$ and
$\Lambda_y \not\subset \Lambda_x$.
\end{lemma}
\begin{proof}
\leavevmode

If $\Lambda_x \subset \Lambda_y$, then $0_x \geq 0_y$ and $1_x \leq 1_y$.  But
the first condition implies by lem.~\ref{lem:(12)-monotone} that $x \geq y$ and
the second implies $x \leq y$.  Together, they imply $x = y$.

Similarly, $\Lambda_y \subset \Lambda_x$ implies $x = y$.
\end{proof}

\begin{definition}
We define $\sup_L(a_1, \ldots, a_n)$ to be the supremum in $L$ of the
elements (in $L$) or subsets (of $L$) $a_1, \ldots, a_n$, when it
exists.
We define $\inf_L(a_1, \ldots, a_n)$ to be the infimum in $L$ of the
elements (in $L$) or subsets (of $L$) $a_1, \ldots, a_n$, when it
exists.
\end{definition}

\begin{lemma}\citeHerr*{(13)}{(13)} \label{lem:(13)}
$\sup_L(0_x, 0_y)$ exists and  $\sup_L(0_x, 0_y) = 0_{x \vee y}$.
\end{lemma}
\begin{proof}
\leavevmode

By lem.~\ref{lem:(12)-polytone}, $0_{x \vee y}$ is an upper bound of $0_x$
and $0_y$.  To show that it is the least upper bound,
assume $e \geq 0_x, 0_y$ in
$L$.  By lem.~\ref{lem:(10)+}, there exists
an ascending sequence $(z_\bullet, a_\bullet)$ with first element
$0_x$, last element $e$, and first block $x$.  Define $z^+$ to be its
last block, so $z^+ \geq x$.
Similarly, there exists
an ascending sequence $(w_\bullet, b_\bullet)$ with first element
$0_y$, last element $e$, and first block $y$.  Define $w^+$ to be its
last block, so $w^+ \geq y$.

Thus, $e \in \Lambda_{z^+} \cap \Lambda_{w^+}$, and by
lem.~\ref{lem:MCd-for-Lam}, $e \in \Lambda_{z^+ \vee w^+}$,
so $e \geq 0_{z^+ \vee w^+}$.
Since $z^+ \vee w^+ \geq x \vee y$,
by lem.~\ref{lem:(12)-polytone},
$e \geq 0_{z^+ \vee w^+} \geq 0_{x \vee y}$.
Since we have shown that any upper bound of $0_x$ and $0_y$ is $\geq 0_{x \vee y}$,
$\geq 0_{x \vee y}$ is the least upper bound.
\end{proof}

\begin{definition} \label{def:vee-wedge-sub-x}
For every $v \in S$ , we define on $L$ the partial
binary operations $\vee_v$ and $\wedge_v$:
$a \vee_v b = \pi_v(\pi_v^{-1}\,a \vee \pi_v^{-1}\, b)$ and
$a \wedge_v b = \pi_v(\pi_v^{-1}\,a \wedge \pi_v^{-1}\, b)$
when $a, b \in \Lambda_v$,
that is, the images of $\vee$ and $\wedge$ on $L_v$ under the map $\pi_v$.
\end{definition}

\begin{lemma} \citeHerr*{(8)}{(8)} \label{lem:(8)-for-Lam}
Suppose $\Lambda_x \cap \Lambda_y \neq \zeroslash$.
Then $\Lambda_x \cap \Lambda_y$ is an interval in $\Lambda_x$ and in $\Lambda_y$.
The operations $\wedge_x$ and $\wedge_y$ are closed on and agree on
$\Lambda_x \cap \Lambda_y$.
The operations $\vee_x$ and $\vee_y$ are also closed on and agree on
$\Lambda_x \cap \Lambda_y$.
\end{lemma}
\begin{proof}
\leavevmode

Lem.~\ref{lem:pi-lattice-iso} shows that $\Lambda_x$ and $\Lambda_y$ are
intervals in $L$, and since they intersect, $\Lambda_x \cap \Lambda_y$
is an interval in $\Lambda_x$ and in $\Lambda_y$.
Define $S_x = \pi_x^{-1}(\Lambda_x \cap \Lambda_y)$ and
$S_y = \pi_y^{-1}(\Lambda_x \cap \Lambda_y)$.
Define $\rho_x = \pi_x|_{S_x}$ and $\rho_y = \pi_y|_{S_y}$.
By lem.~\ref{lem:pi-lattice-iso}, $\rho_x$ is a lattice
isomorphism from $S_x$ to $\Lambda_x \cap \Lambda_y$,
$\rho_y$ is a lattice isomorphism from $S_y$ to
$\Lambda_x \cap \Lambda_y$, and
${\rho_y}^{-1} \circ \rho_x$ is a lattice isomorphism
from $S_x$ to $S_y$.
Thus the operations $\wedge_x$ and $\wedge_y$ are
closed on and agree on $\Lambda_x \cap \Lambda_y$, and
the operations $\vee_x$ and $\vee_y$ are closed on and agree on
$\Lambda_x \cap \Lambda_y$.
\end{proof}

To simplify the proof of lem.~\forwardref{lem:(14)}, we first prove this
hideously technical lemma:

\begin{lemma}(part of \citeHerr*{(14)}{(14)}) \label{lem:(14)-induction}
Given $e \in L$, $w, x \in S$, and $a, b \in \Lambda_x$
such that
\begin{enumerate}
\item $e \in \Lambda_w$;
\item $e > a, b$;
\item there exists an ascending sequence $(t_\bullet, p_\bullet)$
of length $n \geq 1$
with $p_0 = a$, $p_n = e$, $t_1 \geq x$, and $t_n \leq w$; and
\item there exists an ascending sequence $(u_\bullet, q_\bullet)$
of length $m \geq 1$
with $q_0 = b$, $q_m = e$, $u_1 \geq x$, and $u_m \leq w$;
\end{enumerate}
then $e \geq a \vee_x b$.
\end{lemma}
\begin{proof} \disconnect
First we use downward induction on the length $n$ to require
that $p_1 > a$:  If $p_1 = a$, we apply this lemma
inductively to the ascending sequences
$((t_i)_{2 \leq i \leq n}, (p_1)_{1 \leq i \leq n})$ of length $n-1$
and $(u_\bullet, q_\bullet)$.
Since $t_n = e > a$, $n \geq 2$ and the length of the new ascending
sequence remains $\geq 1$, and since the new sequence is always
shorter than the original, the recursion must end.

Similarly, we use downward induction on $m$ to require
and $q_1 > b$.

Necessarily $x \leq t_1 \leq t_n \leq w$.
We prove the lemma by induction on $x \in S$ descending from $w$.
Since $S$ is locally finite \ref{MCf}, the induction will reach all $x \leq w$.
\begin{enumerate}
\item Case $x = w$:
Then $a, b, e \in \Lambda_x$ and
by lem.~\ref{lem:pi-poset-iso}, $e \geq_x a, b$, so
$e \geq_x a \vee_x b$, and so $e \geq a \vee_x b$.
\item Case $x < w$:
This case can in turn be proven inductively by $(a, b)$ descending in
$\Lambda_x \times \Lambda_x$ (which by lem.~\ref{lem:canon-1-1} is
finite). The induction start, $a = b = 1_x$, is
trivial since then $e \geq a = a \vee_x b$.

So let $a, b \in \Lambda_x$ with the induction hypothesis:
\begin{equation*}
\textit{For all $c, d \in \Lambda_x$ with $e \geq c$, $e \geq d$, $c \geq a$,
$d \geq b$, and either
$c > a$ or $d > b$; we have $e \geq c \vee_x d$.} \tag{a}
\end{equation*}
From this we must prove:
\begin{equation*}
e \geq a \vee_x b. \tag{a\supprime}
\end{equation*}
\begin{enumerate}
\item Case $t_1 = x$:
If $t_1 = x$, then $p_1 >_x p_0 = a$ and applying the induction (a) to $p_1$
and $b$ gives us by (a\supprime), $p_1 \vee_x b \leq e$.
Then $a \vee_x b \leq p_1 \vee_x b \leq e$, so $a \vee_x b \leq e$.
\item Case $u_1 = x$:
Analogously, if $u_1 = x$, then $a \vee_x b \leq e$.
\item Case $t_1, u_1 > x$ and $t_1 \wedge u_1 > x$:
Since $x \leq t_1 \wedge u_1 \leq t_1$ and $a \in \Lambda_x, \Lambda_{t_1}$, from
lem.~\ref{lem:MCd-for-Lam}, it follows that $a \in \Lambda_{t_1 \wedge u_1}$.
Similarly, $b \in \Lambda_{t_1 \wedge u_1}$.
We know $x < t_1 \wedge u_1 \leq w$.
Since $t_1 \wedge u_1 \leq t_1$ and $a \leq_{t_1 \wedge u_1} a$,
we can use lem.~\ref{lem:asc-seq-prepend} to construct a new
ascending sequence by prepending element $a$ and block
$t_1 \wedge u_1$ to $(t_\bullet, p_\bullet)$.
Similarly, we can prepend $b$ and $t_1 \wedge u_1$ to
$(u_\bullet, q_\bullet)$.

\hspace{1.25em}Using these two new ascending sequences, we can
use the lemma inductively (replacing $x$ with $t_1 \wedge u_1$) to conclude
that $a \vee_{t_1 \wedge u_1} b \leq e$.
Since $a, b \in \Lambda_x \cap \Lambda_{t_1 \wedge u_1}$, by
lem.~\ref{lem:(1)-for-Lam},
$a \vee_x b \in \Lambda_x \cap \Lambda_{t_1 \wedge u_1}$.
By lem.~\ref{lem:(8)-for-Lam},
$a \vee_x b = a \vee_{t_1 \wedge u_1} b \leq e$.
\item Case $t_1, u_1 > x$ and $t_1 \wedge u_1 = x$:
Since $a \in \Lambda_x \cap \Lambda_{t_1}$ and
$b \in \Lambda_x \cap \Lambda_{u_1}$,
lem.~\ref{lem:Lam-overlap-gamma} ensures
that $0_{t_1}, 0_{u_1} \in \Lambda_x$. Let $c = 0_{t_1} \vee_x 0_{u_1}$.
By lem.~\ref{lem:(1)-for-Lam}, $c \in \Lambda_{t_1}$ and $c \in \Lambda_{u_1}$,
so by lem.~\ref{lem:MCd-for-Lam},
$c \in \Lambda_{t_1} \cap \Lambda_{u_1} = \Lambda_x \cap \Lambda_{t_1 \vee u_1}$.
Hence $c \geq 0_{t_1 \vee u_1}$.
On the other hand, since $\Lambda_x \cap \Lambda_{t_1 \vee u_1} \neq \zeroslash$,
$0_{t_1 \vee u_1} \in \Lambda_x$, and by lem.~\ref{lem:(1)-for-Lam},
$0_{t_1}, 0_{u_1} \leq_x 0_{t_1 \vee u_1}$.
Assembling these, $0_{t_1} \vee_x 0_{u_1} = 0_{t_1 \vee u_1} = c$.

\hspace{1.25em}By lem.~\ref{lem:(13)},
$c = 0_{t_1 \vee u_1} = \sup_L(0_{t_1}, 0_{u_1})$.
We know $x < t_1, u_1 \leq w$.
Choose a saturated chain in $S$,
$t_1 \vee u_1 \lessdot v_1 \lessdot v_2 \lessdot \cdots \lessdot v_k \lessdot w$.
We can construct a new ascending sequence from the chain:
$(r_\bullet, f_\bullet) =
((t_1 \vee u_1, v_1, v_2, \ldots, v_k, w),
(0_{v_1}, 0_{v_2}, \ldots, 0_{v_k}, 0_w))$.
Because the chain is saturated, \ref{MCb1} shows this is a
valid ascending sequence.
we use the lemma inductively
using ascending sequences $(t_\bullet, p_\bullet)$ and $(r_\bullet, f_\bullet)$,
and replacing $x$ with $t_1$ and $b$ with $c$
to show $a \vee_{t_1} c \leq e$.
By lem.~\ref{lem:(8)-for-Lam}, $a \vee_x c = a \vee_{t_1} c \leq e$.

\hspace{1.25em}Similarly we show $b \vee_x c = b \vee_{u_1} c \leq e$.
\begin{enumerate}
\item Case $a \not\geq c$:
If $a \not\geq c$, then also $a \not\geq_x c$ and therefore $a \vee_x c >_x a$.
Applying recursion (a) to $a \vee_x c$ and $b$,
(a\supprime) is $(a \vee_x c) \vee_x (b \vee_x c) \leq e$,
from which immediately follows $a \vee_x b \leq e$.
\item Case $b \not\geq c$:
Similarly, in this case $a \vee_x b \leq e$.
\item Case $a \geq c$ and $b \geq c$:
By lem.~\ref{lem:(12)-polytone},
$c \leq a \leq 1_x \leq 1_{t_1} \leq 1_{t_1 \vee u_1}$
and similarly $c \leq b \leq 1_{t_1 \vee u_1}$, so lem.~\ref{lem:(11)}
shows $a$ and $b$ are elements of $\Lambda_{t_1 \vee u_1}$.
Using lem.~\ref{lem:asc-seq-lift-weak},
from $(t_\bullet, p_\bullet)$ and $(u_\bullet, q_\bullet)$
we can construct new ascending sequences from $a$ to $e$ and $b$ to
$e$, both with first blocks $t_q \vee u_1$.
Since $x < t_1 \vee u_1 \leq w$, we use the lemma inductively
on the two new sequences
(replacing $x$ with $t_1 \vee u_1$) to conclude that
$a \vee_{t_1 \vee u_1} b \leq e$.
However, $a \vee_{t_1 \vee u_1} b = a \vee_x b$ by lem.~\ref{lem:(8)-for-Lam},
so $a \vee_x b \leq e$.
\end{enumerate}
\end{enumerate}
\end{enumerate}
\end{proof}

\begin{subnumbering}

\begin{remark} \label{rem:proof-(14)-is-gross}
\textup{
There clearly should be a simpler proof of lem.~\forwardref{lem:(14)}.
}
\end{remark}

\begin{lemma}\citeHerr*{(14)}{(14)} \label{lem:(14)}
For $x \in S$ and $a, b \in \Lambda_x$, then $\sup_L(a, b)$ exists
and $= a \vee_x b$.
\end{lemma}
\begin{proof} \disconnect
Clearly, by lem.~\ref{lem:pi-poset-iso}, $a \vee_x b \geq a, b$.

Conversely, suppose that $e \geq a, b$.  Then we must show that
$e \geq a \vee_x b$.
If $e = a$, then
since $e \geq b$, $a \geq b$, by lem.~\ref{lem:pi-poset-iso} $a \geq_x b$,
$a = a \vee_x b$, $e = a \vee_x b$, and so $e \geq a \vee_x b$.
Similarly if $e = b$, we show $e \geq a \vee_x b$.

Thus, we can assume $e > a, b$.
By lem.~\ref{lem:(10)+}, there exists
an ascending sequence $(t_\bullet, p_\bullet)$ with first element $a$,
last element $e$, and $t_1 \geq x$.
Similarly, there exists
an ascending sequence $(q_\bullet, u_\bullet)$ of length $m$ with
first element $b$,
last element $e$, and $u_1 \geq x$.
Define $w = t_n \vee u_m$.
By construction, $e \in \Lambda_{t_n}, \Lambda_{u_m}$, so
$t_n, u_m \in \Pi_e$, and by lem.~\ref{lem:Pi-conv-sub},
$w = t_n \vee u_m \in \Pi_e$ and $e \in \Lambda_w$.

We can apply lem.~\ref{lem:(14)-induction} to $a$, $b$, $x$, $e$, $w$,
$(t_\bullet, p_\bullet)$, and $(u_\bullet, q_\bullet)$
to show $e \geq a \vee_x b$.
\end{proof}

\end{subnumbering}

\begin{lemma} \citeHerr*{(15)}{(15)} \label{lem:(15)}
If $x = x_0 \lessdot x_1 \lessdot \cdots \lessdot x_n = y$ is a
saturated chain in $S$ with
$n \geq 0$, $a \in \Lambda_x$, and $b \in \Lambda_y$, then
$\sup_L(a, b)$ exists and
$$ \sup_L(a, b) =
((\cdots((a \vee_{x_0} 0_{x_1}) \vee_{x_1} 0_{x_2}) \vee_{x_2} \cdots)
\vee_{x_{n-1}} 0_{x_n}) \vee_y b. $$
\end{lemma}
\begin{proof}
\leavevmode

The proof proceeds inductively on $n$.
The case $n = 0$ is immediate from lem.~\ref{lem:(14)}.

So let $n \geq 1$. Set $c_0 = a \in \Lambda_{x_0}$.
For $1 \leq i \leq n$, we inductively define $c_i = c_{i-1} \vee_{x_{i-1}} 0_{x_i}$
and inductively prove $c_i \in \Lambda_{x_i}$ using \ref{MCb1}, \ref{MCe},
and $0_{x_{i+1}} \in \Lambda_{x_i}$.
By the induction on $n$, $c_{n-1} = \sup_L(a, 0_{x_{n-1}})$.
Because of lem.~\ref{lem:(14)},
$$ c_n = c_{n-1} \vee_{x_{n-1}} 0_{x_{n-1}}
     = \sup_L(c_{n-1}, 0_{x_{n-1}})
     = \sup_L(a, 0_{x_{n-1}}, 0_{x_n}) = \sup_L(a, 0_{x_n}). $$
We've shown $c_n \in \Lambda_{x_n}$, so $c_n \vee_{x_n} b$ is defined
and by lem.~\ref{lem:(14)}, $c_n \vee_{x_n} b = \sup_L(c_n, b)$.
Therefore $c_n \vee_{x_n} b = \sup_L(a, 0_{x_n}, b) = \sup_L(a, b)$.
\end{proof}

\begin{lemma} \citeHerr*{(16)}{(16)} \label{lem:(16)}
If $a \in \Lambda_x$ and $b \in \Lambda_y$, then
$\sup_L(a, b)$ exists and
$= \sup_L(a, 0_{x \vee y}) \vee_{x \vee y} \sup_L(b, 0_{x \vee y})$.
\end{lemma}
\begin{proof}
\leavevmode

By lem.~\ref{lem:(15)},
$\sup_L(a, 0_{x \vee y}), \sup_L(b, 0_{x \vee y}) \in \Lambda_{x \vee y}$.
Further, by lem.~\ref{lem:(14)} and~\ref{lem:(13)},
$$ \sup_L(a, 0_{x \vee y}) \vee_{x \vee y} \sup_L(b, 0_{x \vee y})
= \sup_L(a, b, 0_{x \vee y})
= \sup_L(a, b, 0_x, 0_y)
= \sup_L(a, b). $$
\end{proof}

Dually, we prove:

\begin{lemma}\citeHerr*{(13*)}{(13*)} \label{lem:(13*)}
$\inf_L(1_x, 1_y)$ exists and  $\inf_L(1_x, 1_y) = 1_{x \wedge y}$.
\end{lemma}

\begin{lemma}\citeHerr*{(14*)}{(14*)} \label{lem:(14*)}
For $x \in S$ and $a, b \in \Lambda_x$, then $\inf_L(a, b)$ exists
and $= a \wedge_x b$.
\end{lemma}

\begin{lemma} \citeHerr*{(15*)}{(15*)} \label{lem:(15*)}
If $x = x_0 \gtrdot x_1 \gtrdot \cdots \gtrdot x_n = y$ is a
saturated chain in $S$ with
$n \geq 0$, $a \in \Lambda_x$, and $b \in \Lambda_y$, then
$\inf_L(a, b)$ exists and
$$ \inf_L(a, b) =
((\cdots((a \wedge_{x_0} 0_{x_1}) \wedge_{x_1} 0_{x_2}) \wedge_{x_2} \cdots)
\wedge_{x_{n-1}} 0_{x_n}) \wedge_y b. $$
\end{lemma}

\begin{lemma} \citeHerr*{(16*)}{(16*)} \label{lem:(16*)}
If $a \in \Lambda_x$ and $b \in \Lambda_y$, then
$\inf_L(a, b)$ exists and
$= \inf_L(a, 0_{x \wedge y}) \wedge_{x \wedge y} \sup_L(b, 0_{x \wedge y})$.
\end{lemma}

\begin{definition}
For $x, y \in L$, define $x \vee y = \sup_L(x, y)$ and
$x \wedge y = \inf_L(x, y)$.
\end{definition}

\begin{theorem} (part of \citeHerr*{Satz~2.1}{Th.~2.1}) \label{th:lattice}
$L$ (with the order relation $\leq$)
is a lattice with $\vee$ and $\wedge$ as the lattice operations.
\end{theorem}
\begin{proof}
\leavevmode

This is immediate from lem.~\ref{lem:(16)} and~\ref{lem:(16*)}.
\end{proof}

\begin{remark}
\textup{
In v1 of this paper, the proof that $L$ is a lattice depends
on monotony, property \ref{MCb2}.
In v2, the proof was sharpened to no longer depend on \ref{MCb2}.
}
\end{remark}

\begin{lemma} (part of \citeHerr*{Zusatz~2.1}{Add.~2.1}) \label{lem:sup-inf-polytone}
The map $x \mapsto 0_x$ is a $\sup$-homomorphism of $S$ into $L$.
The map $x \mapsto 1_x$ is a $\inf$-homomorphism of $S$ into $L$.
\end{lemma}
\begin{proof}
\leavevmode

This is direct from
lem.~\ref{lem:(12)-polytone}, \ref{lem:(13)}, and~\ref{lem:(13*)}.
\end{proof}

\begin{subnumbering}

\begin{lemma}\flagstart$\dagger$ (part of \citeHerr*{Zusatz~2.1}{Add.~2.1}) \label{lem:sup-inf}
The maps $x \mapsto 0_x$ and $x \mapsto 1_x$ are injective.
\end{lemma}
\begin{proof}
\leavevmode

This is direct from lem.~\ref{lem:sup-inf-polytone} and
lem.~\ref{lem:(12)-monotone}.
\end{proof}

\end{subnumbering}

\paragraph{The sum is a modular lattice}

\begin{theorem} \citeHerr*{Satz~3.2}{Th.~3.2} \label{th:modular}
The sum $L$ of a \mcs is a modular lattice.
\end{theorem}
\begin{proof}
\leavevmode

Th.~\ref{th:lattice} has proven that $L$ is a lattice.
To prove it is modular:
Assume $a \lessdot b, c$ in $L$ and $b \neq c$.
Then by lem.~\ref{lem:covers-xfer} there exists $x \in S$ such that
$a, b \in \Lambda_x$.  Similarly there exists $y \in S$ such that
$a, c \in \Lambda_y$.  Since $a \in \Lambda_x \cap \Lambda_y$, by
lem.~\ref{lem:MCd-for-Lam}, $a \in \Lambda_{x \vee y}$, and by
lem.~\ref{lem:(1)-for-Lam}, $b, c \in \Lambda_{x \vee y}$.
By \ref{MCg}, $\Lambda_{x \vee y}$ is modular, so there exists
$d \in \Lambda_{x \vee y}$ such that $b, c \lessdot d$.
Thus, $L$ is upper semimodular.

Dually, we prove $L$ is lower semimodular.

Now assume $a, b, c \in L$.  Define
$N = [a \wedge b \wedge c, a \vee b \vee c]$, which is a sublattice of
$L$.  By the above, $N$ is upper and lower semimodular.
By th.~\ref{th:locally-finite}, $N$ is finite, and so $N$ is modular.
Thus, $a$, $b$, and $c$ satisfy the modular identity in $N$ and thus
in $L$.
This proves that $L$ is modular.
\end{proof}

\paragraph{Existence of \^0}

\begin{theorem}\flagstart$\dagger$ \label{th:sum-0}
The sum lattice $L$ has a minimum element iff the skeleton lattice $S$
has a minimum element.
\end{theorem}
\begin{proof}
\leavevmode

Regarding $\Rightarrow$:
Assume $L$ has a minimum element $\hatzero_L$.
There exists $x \in S$ and $a \in L_x$ such that
$\pi_x\,a = \hatzero_L$.
If $S$ does not have a minimum element, there exists $y \in S$ such
that $y < x$.  By lem.~\ref{lem:(12)-monotone} and~\ref{lem:pi-lattice-iso}
$0_y < 0_x \leq \hatzero_L$, contradicting that $\hatzero_L$ is the minimum
element of $L$.

Regarding $\Leftarrow$:
Assume that $S$ has a minimum element $\hatzero_S$.
If $L$ does not have a minimum element, then
$0_{\hatzero_S}$ is not the minimum element of $L$ and there is an
$a \in L$ such that $a < 0_{\hatzero_S}$.
By lem.~\ref{lem:(10)+}, there is an ascending sequence with first
element $a$ and last element $0_{\hatzero_S}$
Define $x$ to be the first block of the sequence and $y$ to be the
last block, with $x \leq y$ in $S$.
Then $a \in \Lambda_x$ and $0_{\hatzero_S} \in \Lambda_y$, so
$0_{\hatzero_S} \in \Pi_y$ and
by lem.~\ref{lem:0-1-extreme}, $y \leq \hatzero_S$,
requiring $y = \hatzero_S$.
That implies $x = \hatzero_S$ and so $a \geq 0_{\hatzero_S}$,
contradicting that $a < 0_{\hatzero_S}$.
\end{proof}

\paragraph{The adjoint maps \texorpdfstring{$\Phi$ and $\Psi$}%
{\83\246\ and \83\250}}

The family of partial functions $\P{\bullet}{\bullet}$ and their inverses
$\PI{\bullet}{\bullet}$ defined for pairs of indexes $x$ and $y$ with
$x \leq_\gamma y$
can be extended to a family of total functions defined for all pairs
$x \leq y$
in the manner of \cite{DayHerr1988a}*{Def.~4.4}:

\begin{definition} \label{def:adjoint}
For $x \leq y$ in $S$, we define
$\PP{x}{y}: L_x \rightarrow L_y$ and
$\PPI{x}{y}: L_y \rightarrow L_x$:
\begin{alignat*}{2}
\PP{x}{y}\,a & =
\begin{cases}
\P{x}{y}(a \vee \Z{x}{y}) & \rmif x \relgamma y \\
\hatzero_{L_y}               & \rmif x \not\relgamma y
\end{cases} \\
\PPI{x}{y}\,a & =
\begin{cases}
\PI{x}{y}(a \wedge \O{x}{y}) & \rmif x \relgamma y \\
\hatone_{L_x}                    & \rmif x \not\relgamma y
\end{cases}
\end{alignat*}
\end{definition}

\begin{lemma} \label{lem:adjoint-homo}
For $x \leq y$ in $S$,
$\PP{x}{y}$ is a lattice homomorphism from $L_x$ to $L_y$
and $\PPI{x}{y}$ is a lattice homomorphism from $L_y$ to $L_x$.
\end{lemma}

\begin{theorem} \label{th:adjoint}
For $x \leq y$ in $S$,
$\PP{x}{y}$ and $\PPI{x}{y}$ are an adjoint pair of homomorphisms between
$L_x$ and $L_y$.\cite{WikiGal}
\end{theorem}
\begin{proof}
\leavevmode

To show that $\PP{x}{y}$ and $\PPI{x}{y}$ are an adjoint pair of
functions, we must show that for all $a \in L_x$ and $b \in L_y$,
$$ \PP{x}{y}(a) \leq b \textup{ iff } a \leq \PPI{x}{y}(b). $$

If $x \not\relgamma y$, then for any $a \in L_x$ and $b \in L_y$,
$\PP{x}{y}(a) = \hatzero_{L_y} \leq b$ and
$a \leq \hatone_{L_x} =  \PPI{x}{y}(b)$, so their equivalence is
trivial.

If $x \relgamma y$, then the following statements are all
equivalent:
\begin{align*}
\PP{x}{y}(a) & \leq b \\
\interject{by def.~\ref{def:adjoint},}
\P{x}{y}(a \vee \Z{x}{y}) & \leq b \\
\interject{since $\O{x}{y}$ is the maximum of $\I{x}{y}$, the range of
$\P{x}{y}$, this is equivalent to}
\P{x}{y}(a \vee \Z{x}{y}) & \leq b \wedge \O{x}{y} \\
\interject{because $\PI{x}{y}$ is an isomorphism,}
a \vee \Z{x}{y} & \leq \PI{x}{y}(b \wedge \O{x}{y}) \\
\interject{since $\Z{x}{y}$ is the minimum of $\F{x}{y}$, the range of
$\PI{x}{y}$,}
a & \leq \PI{x}{y}(b \wedge \O{x}{y}) \\
\interject{by def.~\ref{def:adjoint},}
a & \leq \PPI{x}{y}(b)
\end{align*}
\end{proof}

\begin{remark}
\textup{
$\PP{x}{y}$ and $\PPI{x}{y}$ are characterized by more than that they
are adjoints.  Inter alia, the image of $\PP{x}{y}$ is an ideal of $L_y$ and
the image of $\PPI{x}{y}$ is a filter of $L_x$.
}
\end{remark}

\begin{lemma} \label{lem:PP-identity}
For $x \in S$,
$\PP{x}{x}$ and $\PPI{x}{x}$ are both the identity map on $L_x$.
\end{lemma}

\begin{lemma}
For $x \leq y$ in $S$,
$\PP{x}{y}\,\hatzero_{L_x} = \hatzero_{L_y}$ and
$\PPI{x}{y}\,\hatone_{L_y} = \hatone_{L_x}$.
\end{lemma}

$\PP{\bullet}{\bullet}$ and $\PPI{\bullet}{\bullet}$
extend the corresponding
$\P{\bullet}{\bullet}$ and $\PI{\bullet}{\bullet}$:

\begin{lemma}
For $x \leq_\gamma y$ in $S$,
$\PP{x}{y}|_{\F{x}{y}} = \P{x}{y}$ and
$\PPI{x}{y}|_{\I{x}{y}} = \PI{x}{y}$.
\end{lemma}

\begin{lemma} \label{lem:adjoint-compose}
For $x \leq y \leq z$ in $S$,
$\PP{x}{z} = \PP{y}{z} \circ \PP{x}{y}$ and
$\PPI{x}{z} = \PPI{y}{x} \circ \PPI{z}{y}$.
\end{lemma}
\begin{proof}
\leavevmode

The proof of $\PP{x}{z} = \PP{y}{z} \circ \PP{x}{y}$ has several
cases.  Assume $a \in L_x$.

If $x \not\relgamma y$ (and consequently $x \not\relgamma z$):
$\PP{x}{z}\,a = \hatzero_{L_z} = \PP{y}{z}\,\hatzero_{L_y} =
\PP{y}{z}(\PP{x}{y}\,a)$.

If $y \not\relgamma z$ (and consequently $x \not\relgamma z$):
$\PP{x}{z}\,a = \hatzero_{L_z} = \PP{y}{z}(\PP{x}{y}\,a)$.

If $y \relgamma x,z$ but $x \not\relgamma z$:
\begin{alignat*}{2}
\PP{x}{z}\,a & = \hatzero_{L_z} \\
& = \P{y}{z}(\Z{y}{z}) \\
& = \PP{y}{z}(\Z{y}{z}) \\
\interject{since $x \not\relgamma z$, by \ref{MCh}, $\I{x}{y}$,
the range of $\P{x}{y}$, is disjoint from $\F{y}{z}$, so
$\P{x}{y}\,a \leq \O{x}{y} < \Z{y}{z}$,}
& = \PP{y}{z}(\P{x}{y}\,a \vee \Z{y}{z}) \\
& = \PP{y}{z}(\PP{x}{y}\,a)
\end{alignat*}

If $x \relgamma z$ (and consequently $y \relgamma x,z$):
\begin{alignat*}{2}
\PP{x}{z}\,a
& = \P{x}{z}(\Z{x}{z} \vee a) \\
\interject{by \ref{MCc},}
& = \P{y}{z}(\P{x}{y}(\Z{x}{z} \vee a)) \\
\interject{because $\Z{x}{z} \geq \Z{x}{y}$,}
& = \P{y}{z}(\P{x}{y}(\Z{x}{z} \vee \Z{x}{y} \vee a)) \\
\interject{because $\P{x}{y}$ is an isomorphism,}
& = \P{y}{z}(\P{x}{y}(\Z{x}{z}) \vee \P{x}{y}(\Z{x}{y} \vee a)) \\
\interject{because $\P{x}{y}(\Z{x}{z}) = \Z{y}{z}$,}
& = \P{y}{z}(\Z{y}{z} \vee \P{x}{y}(\Z{x}{y} \vee a)) \\
& = \PP{y}{z}(\PP{x}{y}\,a)
\end{alignat*}

We prove $\PPI{x}{z} = \PPI{y}{x} \circ \PPI{z}{y}$ dually.
\end{proof}

The adjoint functions allow us to restate
lem.~\ref{lem:(15)}, \ref{lem:(16)}, \ref{lem:(15*)}, and~\ref{lem:(16*)}
in this form:

\begin{lemma} \label{lem:(15)-for-adjoint}
If $x \leq y$ in $S$,
$a \in \Lambda_x$, and $b \in \Lambda_y$, then
\begin{equation*}
a \vee b = \pi_y(\PP{x}{y}(\pi_x^{-1}\,a) \vee \pi_y^{-1}\,b).%
\footnote{Note that the first $\vee$ is in $L$ and the second is in $L_y$.}
\end{equation*}
Or equivalently: If $c \in L_x$, and $d \in L_y$, then
$$ \pi_x\,c \vee \pi_y\,d = \pi_y(\PP{x}{y}\,c \vee d). $$
\end{lemma}
\begin{proof}
\leavevmode

Regarding the first statement:

Choose a saturated chain in $S$ from $x$ to $y$,
$x = x_0 \lessdot x_1 \lessdot \cdots \lessdot x_n = y$
for some $n \geq 0$.
By lem.~\ref{lem:(15)},
\begin{equation*}
a \vee b =
((\cdots((a \vee_{x_0} 0_{x_1}) \vee_{x_1} 0_{x_2}) \vee_{x_2} \cdots)
\vee_{x_{n-1}} 0_{x_n}) \vee_y b. \tag{a}
\end{equation*}
We prove the lemma by induction on $n$.

If $n=0$, then $x = x_0 = x_n = y$, and (a) reduces to
\begin{alignat*}{2}
a \vee b & = a \vee_y b \\
\interject{by def.~\ref{def:vee-wedge-sub-x} and $x = y$,}
& = \pi_y(\pi_x^{-1}\,a \vee \pi_y^{-1}\,b) \\
\interject{by lem.~\ref{lem:PP-identity}, $\PP{x}{y}$ is the identity,}
& = \pi_y(\PP{x}{y}(\pi_x^{-1}\,a) \vee_y \pi_y^{-1}\,b)
\end{alignat*}

If $n > 0$, showing more terms of (a),
\begin{alignat*}{2}
a \vee b
& = (((\cdots((a \vee_{x_0} 0_{x_1}) \vee_{x_1} 0_{x_2}) \vee_{x_2} \cdots)
\vee_{x_{n-2}} 0_{x_{n-1}}) \vee_{x_{n-1}} 0_{x_n}) \vee_y b \\
\interject{since $0_{x_{n-1}}$ is the identity of $\vee_{x_{n-1}}$,}
& = ((((\cdots((a \vee_{x_0} 0_{x_1}) \vee_{x_1} 0_{x_2}) \vee_{x_2} \cdots)
\vee_{x_{n-2}} 0_{x_{n-1}}) \vee_{x_{n-1}} 0_{x_{n-1}}) \vee_{x_{n-1}} 0_{x_n})
\vee_y b \\
\interject{using lem.~\ref{lem:(15)} on $a \vee 0_{x_{n-1}}$,}
& = ((a \vee 0_{x_{n-1}}) \vee_{x_{n-1}} 0_y) \vee_y b \\
\interject{defining $z = x_{n-1}$ and using $y = x_n$,}
& = ((a \vee 0_{z}) \vee_{z} 0_y) \vee_y b \\
\interject{using the induction hypothesis
on $a \vee 0_z$, since $0_z \in \Lambda_z = \Lambda_{x_{n-1}}$,}
& = (\pi_z(\PP{x}{z}(\pi_x^{-1}\,a) \vee \pi_z^{-1}\,0_z) \vee_z 0_y) \vee_y b \\
\interject{using def.~\ref{def:vee-wedge-sub-x} on $\vee_z$,}
& = \pi_z((\PP{x}{z}(\pi_x^{-1}\,a) \vee \pi_z^{-1}\,0_z) \vee \pi_z^{-1}\,0_y)
 \vee_y b \\
\interject{since $\pi_z^{-1}\,0_z = \hatzero_{L_z}$ and
$\PP{x}{z}(\pi_x^{-1}\,a) \in L_z$,}
& = \pi_z(\PP{x}{z}(\pi_x^{-1}\,a) \vee \pi_z^{-1}\,0_y) \vee_y b \\
\interject{since $\pi_z^{-1}\,0_y = \Z{z}{y}$,}
& = \pi_z(\PP{x}{z}(\pi_x^{-1}\,a) \vee \Z{z}{y}) \vee_y b \\
\interject{since $\PP{x}{z}(\pi_x^{-1}\,a) \vee \Z{z}{y} \in \F{z}{y}$
and $\pi_z = \pi_y \circ \P{z}{y}$ on $\F{z}{y}$,}
& = \pi_y(\P{z}{y}(\PP{x}{z}(\pi_x^{-1}\,a) \vee \Z{z}{y})) \vee_y b \\
\interject{by def.~\ref{def:adjoint},}
& = \pi_y(\PP{z}{y}(\PP{x}{z}(\pi_x^{-1}\,a))) \vee_y b \\
\interject{by lem.~\ref{lem:adjoint-compose},}
& = \pi_y(\PP{x}{y}(\pi_x^{-1}\,a)) \vee_y b \\
\interject{using def.~\ref{def:vee-wedge-sub-x} on $\vee_y$,}
& = \pi_y(\PP{x}{y}(\pi_x^{-1}\,a) \vee \pi_y^{-1}\,b) \\
\end{alignat*}

Regarding the second statement:
This follows from the first statement by setting $a = \pi_x\,c$ and
$b = \pi_y\,d$.
\end{proof}

\begin{lemma} \label{lem:(16)-for-adjoint}
If $a \in \Lambda_x$, and $b \in \Lambda_y$, then
\begin{equation*}
a \vee b =
\pi_{x \vee y}(\PP{x}{x \vee y}(\pi_x^{-1}\,a) \vee
\PP{y}{x \vee y}(\pi_y^{-1}\,b)).%
\footnote{Note that the first $\vee$ is in $L$ and the second $\vee$
is in $L_{x \vee y}$.}
\end{equation*}
Or equivalently:  If $c \in L_x$, and $d \in L_y$, then
$$ \pi_x\,c \vee \pi_y\,d =
\pi_{x \vee y}(\PP{x}{x \vee y}\,c \vee \PP{x}{x \vee y}\,d). $$
\end{lemma}
\begin{proof}
\leavevmode

Regarding the first statement:
By lem.~\ref{lem:(16)},
\begin{alignat*}{2}
a \vee b
& = (a \vee 0_{x \vee y}) \vee_{x \vee y} (b \vee 0_{x \vee y}) \\
\interject{applying lem.~\ref{lem:(15)-for-adjoint} to both terms,}
& = \pi_{x \vee y}(\PP{x}{{x \vee y}}(\pi_x^{-1}\,a) \vee
\pi_{x \vee y}^{-1}\,0_{x \vee y}) \vee_{x \vee y}
\pi_{x \vee y}(\PP{y}{{x \vee y}}(\pi_y^{-1}\,b) \vee
\pi_{x \vee y}^{-1}\,0_{x \vee y}) \\
\interject{since $\pi_{x \vee y}^{-1}\,0_{x \vee y} = \hatzero_{L_{x \vee y}}$
and the ranges of $\PP{x}{x \vee y}$ and $\PP{y}{x \vee y}$ are
$\subset L_{x \vee y}$,}
& = \pi_{x \vee y}(\PP{x}{{x \vee y}}(\pi_x^{-1}\,a)) \vee_{x \vee y}
\pi_{x \vee y}(\PP{y}{{x \vee y}}(\pi_y^{-1}\,b)) \\
\interject{using def.~\ref{def:vee-wedge-sub-x} on $\vee_{x \vee y}$,}
& = \pi_{x \vee y}(\PP{x}{{x \vee y}}(\pi_x^{-1}\,a) \vee
\PP{y}{{x \vee y}}(\pi_y^{-1}\,b))
\end{alignat*}

Regarding the second statement:
This follows from the first statement by setting $a = \pi_x\,c$ and
$b = \pi_x\,d$.
\end{proof}

Dually, we prove:

\begin{lemma} \label{lem:(15*)-for-adjoint}
If $x \leq y$ in $S$,
$a \in \Lambda_x$, and $b \in \Lambda_y$, then
\begin{equation*}
a \wedge b = \pi_x(\pi_x^{-1}\,a \wedge \PPI{x}{y}(\pi_y^{-1}\,b)).
\end{equation*}
Or equivalently: If $c \in L_x$, and $d \in L_y$, then
$$ \pi_x\,c \wedge \pi_y\,d = \pi_x(c \wedge \PPI{x}{y}\,d). $$
\end{lemma}

\begin{lemma} \label{lem:(16*)-for-adjoint}
If $a \in \Lambda_x$, and $b \in \Lambda_y$, then
\begin{equation*}
a \wedge b =
\pi_{x \wedge y}(\PPI{x \wedge y}{x}(\pi_x^{-1}\,a) \wedge
\PPI{x \wedge y}{y}(\pi_y^{-1}\,b)).
\end{equation*}
Or equivalently:  If $c \in L_x$, and $d \in L_y$, then
$$ \pi_x\,c \wedge \pi_y\,d =
\pi_{x \wedge y}(\PPI{x \wedge y}{x}\,c \wedge \PPI{x \wedge y}{x}\,d). $$
\end{lemma}

\begin{subnumbering}

\begin{lemma} \label{lem:leq-for-adjoint}
For any $a, b \in L$ with $a \in \Lambda_x$ and $b \in \Lambda_y$,
$a \leq b$ iff
\begin{equation*}
x \vee y \relgamma y \textup{, }
\pi_y^{-1}\,b \in \F{y}{x \vee y} \textup{, and }
\PP{x}{x \vee y}(\pi_x^{-1}\,a) \leq
\P{y}{x \vee y}(\pi_y^{-1}\,b).
\end{equation*}
Or equivalently,
For any $c \in L_x$ and $d \in L_y$,
$\pi_x\,c \leq \pi_y\,d$ iff
\begin{equation*}
x \vee y \relgamma y \textup{, }
d \in \F{y}{x \vee y} \textup{, and }
\PP{x}{x \vee y}\,c \leq
\P{y}{x \vee y}\,d.
\end{equation*}
\end{lemma}
\begin{proof}
Define $c = \pi_x^{-1}\,a$ and $d = \pi_y^{-1}\,b$.
The following statements are all equivalent:
\begin{alignat*}{2}
a & \leq b \\
a \vee b & = b \\
\interject{by lem.~\ref{lem:(16)-for-adjoint},}
\pi_{x \vee y}(\PP{x}{x \vee y}\,c \vee \PP{y}{x \vee y}\,d) & = b \\
\interject{since the left side is $\in \Lambda_{x \vee y}$ and the
right side is $\in \Lambda_y$, by def.~\ref{def:sim},}
x \vee y \relgamma y \textrm{, }
d \in \F{y}{x \vee y} \textrm{, and }
\PP{x}{x \vee y}\,c \vee
\PP{y}{x \vee y}\,d & = \P{y}{x \vee y}\,d \\
\interject{by def.~\ref{def:adjoint} and that $\Z{y}{x \vee y}$
is the minimum of $\F{y}{x \vee y}$,}
x \vee y \relgamma y \textrm{, }
d \in \F{y}{x \vee y} \textrm{, and }
\PP{x}{x \vee y}\,c \vee \P{y}{x \vee y}\,d
& = \P{y}{x \vee y}\,d \\
x \vee y \relgamma y \textrm{, }
d \in \F{y}{x \vee y} \textrm{, and }
\PP{x}{x \vee y}\,c & \leq \P{y}{x \vee y}\,d \\
\end{alignat*}
\end{proof}

Dually, we prove:

\begin{lemma} \label{lem:geq-for-adjoint}
For any $a, b \in L$ with $a \in \Lambda_x$ and $b \in \Lambda_y$,
$a \leq b$ iff
\begin{equation*}
x \wedge y \relgamma x \textup{, }
\pi_x^{-1}\,a \in \I{x \wedge y}{x} \textup{, and }
\PI{x \wedge y}{x}(\pi_x^{-1}\,a) \leq
\PPI{x \wedge y}{y}(\pi_y^{-1}\,b).
\end{equation*}
Or equivalently,
For any $c \in L_x$ and $d \in L_y$,
$\pi_x\,c \leq \pi_y\,d$ iff
\begin{equation*}
x \wedge y \relgamma x \textup{, }
c \in \I{x \wedge y}{x} \textup{, and }
\PI{x \wedge y}{x}\,c \leq
\PPI{x \wedge y}{y}\,d.
\end{equation*}
\end{lemma}

\end{subnumbering}

The import of
lem.~\ref{lem:(16)-for-adjoint}, \ref{lem:(16*)-for-adjoint}, \ref{lem:leq-for-adjoint}, and~\ref{lem:geq-for-adjoint}
are that the
operations in $L$ can be divided into a ``global'' part that depends
only on the blocks the elements are contained in, and a ``local'' part
relating elements within each block.

\paragraph{Gluing is natural}
Gluing is ``natural'' in that it is compatible with
isomorphism of \mcss (see def.~\ref{def:mcs-iso}) and isomorphism of
lattices.

\begin{theorem} \label{th:glue-natural}
Gluing is compatible with isomorphism of \mcss and isomorphism of
lattices:  If two \mcss are isomorphic, their sum lattices are
isomorphic.
\end{theorem}
\begin{proof}\hspace{-\labelsep}%
\footnote{It seems like there ought to be a metatheorem that takes
note of how the sum of a \mcs is defined and immediately concludes
that the process is natural.  But I do not know of one.}
\leavevmode

Assume we have two \mcss,
$\scrC$ (comprised of skeleton $S$, overlap tolerance $\relgamma$,
blocks $L_\bullet$, connections $\P{\bullet}{\bullet}$,
connection sources
$\F{\bullet}{\bullet}$, and connection targets
$\I{\bullet}{\bullet}$) and
$\scrC^\prime$ (comprised of skeleton $S^\prime$, overlap
tolerance $\relgamma^\prime$,
blocks $L_\bullet^\prime$, connections
$\P{\bullet}{\bullet}^\prime$, connection sources
$\F{\bullet}{\bullet}^\prime$, and connection targets
$\I{\bullet}{\bullet}^\prime$).
Thus for $\scrC$ we have derived objects $M$, $\sim$, $\kappa$,
$\pi_\bullet$, and its sum $L$,
and for $\scrC^\prime$ derived objects $M^\prime$, $\sim^\prime$,
$\kappa^\prime$, $\pi^\prime_\bullet$ and its sum $L^\prime$.

Assume we have an isomorphism $\chi$ from $\scrC$ to $\scrC^\prime$.
Thus, $\chi_S$ is a lattice isomorphism from $S$ to $S^\prime$ and
$(\chi_{Bx})_{x \in S}$ is a family of lattice isomorphisms, each from $L_x$ to
$L^\prime_{\chi_S(x)}$.  Then immediately from the definitions we
have:
\begin{enumerate}
\item From def.~\ref{def:sim}, for $x \in S$ and $a \in \bigcup_{w \in S} L_w$,
$(x,a) \in M$ iff $(\chi_S\,x, \chi_{Bx}\,a) \in M^\prime$.
\item From def.~\ref{def:sim}, for $(x, a), (y, b) \in M$,
$(x, a) \sim (y, b)$ iff
$(\chi_S\,x, \chi_{Bx}\,a) \sim^\prime (\chi_S\,y, \chi_{By}\,b)$.
\item From def.~\ref{def:kappa}, for $(x, a), (y, b) \in M$,
$\kappa\,(x,a) = \kappa\,(y,b)$ iff
$\kappa^\prime\,(\chi_S\,x, \chi_{Bx}\,a) =
\kappa^\prime\,(\chi_S\,y,\chi_{By},b)$.
\item From def.~\ref{def:pi}, for $(x, a), (y, b) \in M$,
$\pi_x\,a = \pi_y\,b$ iff
$\pi^\prime_{\chi_S\,x}\,\chi_{Bx}\,a =
\pi^\prime_{\chi_S\,y}\,\chi_{By}\,b$.
\end{enumerate}

For $(x, a), (y, b) \in M$, we define
$\chi_L(\kappa((x,a))) = \kappa^\prime((\chi_S\,x,\chi_{Bx}\,a))$
from $L$ to $L^\prime$.
Because of (3) and since by def.~\ref{def:sum},
$L$ is $M$ modulo $\sim$ and $L^\prime$ is
$M^\prime$ modulo $\sim^\prime$, $\chi_L$ is well-defined as a
function from $L$ to $L^\prime$:
$\kappa^\prime((\chi_S\,x,\chi_{Bx}\,a))$ is the same for all $(x,a)$
that are $\sim$.
Since $\chi_S$ and the $\chi_{B\bullet}$ are bijections,
$\chi_L(\kappa((\chi_S^{-1}\,x^\prime,\chi_{B\chi_S^{-1}(x)}^{-1}\,a^\prime))) =
\kappa^\prime((x^\prime,a^\prime))$
for all $(x^\prime, a^\prime) \in M$, showing that $\chi_L$ is bijective.

Then it is straightforward that for any $a, b \in L$,
there exists an ascending sequence $(x_\bullet, a_\bullet)$ with
$a = a_0$ and $b = a_n$ iff
there exists an ascending sequence $(x_\bullet^\prime, a_\bullet^\prime)$ with
$\chi_L\,a = a_0^\prime$ and $\chi_L\,b = a_n^\prime$.
From this, it follows from def.~\ref{def:L-leq}
that for any $a, b \in L$, $a \leq b$ iff
$\chi_L\,a \leq \chi_L\,b$, which shows that $L$ and $L^\prime$ are
isomorphic as posets and consequently isomorphic as lattices.
\end{proof}

\section{Dissection}
The dissection of a lattice into maximal complemented blocks was
introduced in Herrmann \citeHerr*{sec.~6}{sec.~6}.  We mostly use
Haiman's notation \cite{Haim1991a}*{sec.~1}.

In this section, let $L$ be a modular, \lffc lattice.

\paragraph{Basic properties of dissection}

\begin{definition} \citeHerr*{sec.~6}{sec.~6} \label{def:star}
For $x \in L$:
\begin{equation*}
x^* = \begin{cases}
\hatone_L                 & \text{if $\hatone_L$ exists and $a = \hatone_L$} \\
\bigvee_{y, x \lessdot y} y & \text{otherwise} \\
\end{cases}
\end{equation*}
\begin{equation*}
x_* = \begin{cases}
\hatzero_L              & \text{if $\hatzero_L$ exists and $a = \hatzero_L$} \\
\bigwedge_{y, y \lessdot x} y & \text{otherwise}
\end{cases}
\end{equation*}
\end{definition}

\begin{lemma} \label{lem:star}
$x^*$ and $x_*$ are well-defined.
\end{lemma}
\begin{proof}
\leavevmode

That $x^*$ and $x_*$ are well-defined follows from $L$ having finite covers.
\end{proof}

\begin{lemma}\citeHerr*{Lem.~6.1}{Lem.~6.1} \label{lem:star-props}
For all $a, b \in L$:
\begin{align*}
& \textrm{(a) If $a \leq b$, then $a^* \leq b^*$.} &
    & \textrm{(a\supdelta) If $a \leq b$, then $a_* \leq b_*$.} \\
& \textrm{(b) $a \leq {a_*}^* \leq a^*$.} &
    & \textrm{(b\supdelta) $a_* \leq {a^*}_* \leq a$.} \\
& \textrm{(c) ${{a^*}_*}^* = a^*$.} &
    & \textrm{(c\supdelta) ${{a_*}^*}_* = a_*$.} \\
& \textrm{(d) If $a = {a^*}_*$ and $b = {b^*}_*$,
        then $a \vee b = {(a \vee b )^*}_*$.} &
    & \textrm{(d\supdelta) If $a = {a_*}^*$ and $b = {b_*}^*$,
        then $a \wedge b = {(a \wedge b)_*}^*$.} \\
& \textrm{(e) $a_* \vee b_* = (a \vee b)_*$.} &
    & \textrm{(e\supdelta) $a^* \wedge b^* = (a \wedge b)^*$.}
\end{align*}
\end{lemma}
\begin{proof}
\leavevmode

Items (a) through (e\supdelta) are proved
in \citeHerr*{Lem.~6.1}{Lem.~6.1}.
\end{proof}

\begin{remark}
\textup{
Although $\bullet^*$ and $\bullet_*$ resemble a Galois connection
between $L$ and itself, in general they are not:\cite{WikiGal}
Define $L = \{0,1,2\} \times \{0,1\}$, the product of a 2-chain and a
1-chain.  Define $x = (0,1)$ and $y = (2,1)$.
Then $x_* = (0,0)$, $x^* = (1,1)$, $y_* = (1,0)$, and $y^* = (2,1)$.
Thus $x_* = (1,1) \leq (2,2) = y$ but
$x = (0,1) \not\leq (1,0) = y_*$, showing the two operations are not
adjoints.
}
\end{remark}

\begin{lemma}\citeHerr*{sec.~6}{sec.~6} \label{lem:blocks-modular}
For any $x$, the intervals $[x, x^*]$ and $[x_*, x]$ are finite,
modular, and complemented.
\end{lemma}
\begin{proof}
\leavevmode

These intervals are modular because $L$ is modular.
They are finite because $L$ is locally finite.
For $[x, x^*]$, the maximum of the interval is a join of atoms and
for $[x_*, x]$, the minimum of the interval is a meet of coatoms,
so in either case th.~\ref{th:complemented} can be applied to show
that the interval is complemented.
\end{proof}

\begin{definition} \label{def:block}
A \emph{block} of $L$ is a maximal (under inclusion) complemented
interval of $L$.  The \emph{dissection skeleton} $S$ of $L$ is the set of all
blocks of $L$.  Given a block $B \in S$, $0_B$ is the minimum element of the
block and $1_B$ is the maximum element of the block.
\end{definition}

\begin{lemma} \label{lem:block-distinct}
If two blocks are distinct, their minimum elements are distinct and
their maximum elements are distinct.
\end{lemma}
\begin{proof}
\leavevmode

Let $B$ and $C$ be two blocks with $0_B = 0_C$.
Since $B$ is complemented, by th.~\ref{th:complemented}, $1_B$ is the
join of atoms of $B$, and so $1_B \leq {0_B}^*$
Since $B$ is a maximal complemented interval and $[0_B, {0_B}^*]$
is a complemented interval, $1_B = {0_B}^*$.

Similarly we show $1_C = {0_C}^*$.  But since $0_B = 0_C$, this
implies $1_B = 1_C$ and thus $B$ and $C$ are identical.

Dually, we show that $1_B = 1_C$ implies $B$ and $C$ are identical.
\end{proof}

\begin{lemma} \citeHerr*{Satz~6.2 and~6.4}{Th.~6.2 and~6.4} \label{lem:diss-skel}
~	
\begin{enumerate}
\item[(1)] $[x, y]$ is a block iff $x^* = y$ and $y_* = x$.
\item[(2)] The set of minimum elements of blocks is a lattice under the join
operation $x \vee y$ and the meet operation ${(x \wedge y)^*}_*$.
Consequently, this set is a sub-$\sup$-semilattice (and thus sub-poset) of $L$.
\item[(3)] The set of maximum elements of blocks is a lattice under the join
operation ${(x \vee y)_*}^*$ and the meet operation $x \wedge y$.
Consequently, this set is a sub-$\inf$-semilattice (and thus sub-poset) of $L$.
\item[(4)] The set of minimum elements and the set of maximum elements are
isomorphic lattices under the mutually inverse isomorphisms
$x \mapsto x^*$ and $x \mapsto x_*$.
\end{enumerate}
\end{lemma}
\begin{proof}
\leavevmode

These are demonstrated in
\citeHerr*{Satz~6.2, Lem.~6.3, and Satz~6.4}{Th.~6.2, Lem.~6.3, and Th.~6.4}
\end{proof}

\begin{definition} \label{def:block-ops}
We define $\leq$, $\vee$, and $\wedge$ on blocks in $S$
as the corresponding operations on the minimum elements of the blocks, which by
lem.~\ref{lem:diss-skel}(4) are identical to the corresponding operations on
the maximum elements of the blocks.
\end{definition}

\begin{lemma}
$S$ is a lattice under the operations defined in
def.~\ref{def:block-ops}.
\end{lemma}

\begin{definition} \label{def:diss-tol}
The dissection tolerance%
\footnote{\cite{DayHerr1988a} calls the dissection tolerance the
``glue tolerance''.}
$\relgamma$ on $S$ is the relationship between blocks $B$ and
$C$, $B \relgamma C$ iff $B \cap C \neq \emptyset$.
\end{definition}

\begin{lemma} \label{lem:diss-tol}
The dissection tolerance $\relgamma$ of $L$ is a tolerance on $S$.
\end{lemma}
\begin{proof}
\leavevmode

The dissection tolerance is clearly reflexive and symmetric.  It remains to
be shown that it is compatible with the meet and join operations on $S$.

Let there be four blocks $P$, $Q$, $R$, and $S$ for which
$P \relgamma Q$ and $R \relgamma S$. Then there exist $x, y \in L$ for
which $x \in [0_P,1_P] \cap [0_Q,1_Q]$ and
$y \in [0_R,1_R] \cap [0_S,1_S]$.

Then $0_P \leq x \leq 1_P$ and $0_R \leq y \leq 1_R$, so
$0_P \vee 0_R \leq x \vee y \leq 1_P \vee 1_R
\leq {(1_P \vee 1_R)_*}^* = 1_{P \vee R}$ by
lem.~\ref{lem:star-props}(b) and~\ref{lem:diss-skel}(3).
Thus $x \vee y \in [0_{P \vee R}, 1_{P \vee R}]$.

Similarly $0_Q \leq x \leq 1_Q$ and $0_S \leq y \leq 1_S$, so
$0_Q \vee 0_S \leq x \vee y \leq 1_Q \vee 1_S
\leq {(1_Q \vee 1_S)_*}^* = 1_{Q \vee S}$.
Thus $x \vee y \in [0_{Q \vee S}, 1_{Q \vee S}]$.

This shows that $P \vee R$ and $Q \vee S$ both contain $x \vee y$ and
so $P \vee R \relgamma Q \vee S$.

Dually, we prove that
$P \wedge R \relgamma Q \wedge S$.  These show that $\relgamma$
is compatible with meet and join.
\end{proof}

\begin{definition} \label{def:block-phi}
For any blocks $B$, $C$ with $B \leq C$
and $B \cap C \neq \emptyset$, define $\F{B}{C} = \I{B}{C} = B \cap C$.
Define $\P{B}{C}$ as the identity map from $\F{B}{C}$ to $\I{B}{C}$.
\end{definition}

\begin{lemma} \label{lem:block-interval}
For any blocks $B$, $C$ with $B \leq C$
and $B \cap C \neq \emptyset$, $\F{B}{C} = \I{B}{C} = [0_C, 1_B]$, which is
a filter of $B$ and an ideal of $C$.
\end{lemma}
\begin{proof}
\leavevmode

Given that $B \leq C$, we know $0_B \leq 0_C$ and $1_B \leq 1_C$.
That implies
\begin{alignat*}{2}
B \cap C
& = \{ x \in L \mvert 0_B \leq x \leq 1_B \textrm{ and }
0_C \leq x \leq 1_C \} \\
& = \{ x \in L \mvert 0_C \leq x \leq 1_B \} \\
& = [0_C, 1_B]
\end{alignat*}
\end{proof}

\begin{lemma} \label{lem:block-identity}
For any block $B$, $\F{B}{B} = \I{B}{B} = B$ and $\P{B}{B}$ is the
identity map from $B$ to itself.
\end{lemma}

\begin{lemma} \label{lem:block-cover}
For blocks $B \lessdot C$ in $S$, $B \cap C \neq \emptyset$,
$B \not\subset C$, and $C \not\subset B$.
\end{lemma}
\begin{proof}
\leavevmode

By lem.~\ref{lem:block-distinct}, $0_B < 0_C$ and $1_B < 1_C$.
Thus $0_B \in B$ but $0_B \not\in C$, showing $B \not\subset C$.
And  $1_C \in C$ but $1_C \not\in B$, showing $C \not\subset B$.

Let $x$ be a minimal element in $[0_B, 0_C]$ for which
${x^*}_* > 0_B$. Necessarily ${x^*}_* = x$ as otherwise we would have
${x^*}_* < x$, which would allow $y = {x^*}_*$ to be a smaller element
with ${y^*}_* = {{{x^*}_*}^*}_* = {x^*}_* > 0_B$.

Let $X$ be the block $[x, x^*]$.  By construction, $B < X \leq C$ and
since $B \lessdot C$, $C = X$.
Since $x > 0_B$,
choose $y$ such that $0_B \leq y \lessdot x$.
$y^* = {{y^*}_*}^* = {0_B}^* = 1_B$, so every upward cover of $y$ is
$\leq 1_B$.  This implies $x \in B$.
Since $x \in X = C$, $x \in B \cap C$ and $B \cap C \neq \emptyset$.
\end{proof}

\begin{lemma} \label{lem:skeleton-lffc}
$S$ is a \lffc lattice.
\end{lemma}
\begin{proof}
\leavevmode

Because $S$ is a subposet of $L$ and $L$ is locally finite, $S$ is
locally finite.

Given any $B \in S$, for each $C \gtrdot B$ in $S$,
$B \cap C \neq \zeroslash$ by lem.~\ref{lem:block-cover}.
By lem.~\ref{lem:block-interval}, $0_C \in B$.
Since every distinct $C$ has a distinct $0_C$ (by
lem.~\ref{lem:block-distinct}) and $B$ is finite, there are a finite
number of distinct blocks $C \gtrdot B$.
Thus, $S$ has finite upward covers.

Dually, we prove $S$ has finite downward covers.
\end{proof}

\begin{lemma} \label{lem:block-mid-intersect}
For blocks $X$, $Y$, and $Z$ with $X \leq Z \leq Y$,
$X \cap Y \subset Z$ and $X \cap Y = (X \cap Z) \cap (Z \cap Y)$.
\end{lemma}
\begin{proof}
\leavevmode
\begin{alignat*}{2}
0_Z \leq 0_Y & \textrm{ and } 1_X \leq 1_Z \\
0_Z \leq 0_X \vee 0_Y & \textrm{ and } 1_X \wedge 1_Y \leq 1_Z \\
[0_X \vee 0_Y, 1_X \wedge 1_Y] & \subset [0_Z, 1_Z] \\
[0_X, 1_X] \cap [0_Y, 1_Y] & \subset [0_Z, 1_Z] \\
X \cap Y & \subset Z \\
X \cap Y & = X \cap Z \cap Y \\
X \cap Y & = (X \cap Z) \cap (Z \cap Y)
\end{alignat*}
\end{proof}

\begin{lemma} \label{lem:block-mj}
For blocks $X$ and $Y$,
$X \cap Y = (X \wedge Y) \cap (X \vee Y)$
where $X \vee Y$ and $X \wedge Y$ are the join and meet operations
in $S$.
\end{lemma}
\begin{proof}
\leavevmode

\begin{alignat*}{2}
X \cap Y & = [0_X, 1_X] \cap [0_Y, 1_Y] \\
& = [0_X \vee 0_Y, 1_X \wedge 1_Y] \\
\interject{by lem.~\ref{lem:diss-skel}(2)(3),}
& = [0_{X \vee Y}, 1_{X \wedge Y}] \\
\interject{since $X \wedge Y \leq X \vee Y$ and
def.~\ref{def:block-ops} imply $0_{X \wedge Y} \leq 0_{X \vee Y}$ and
$1_{X \wedge Y} \leq 1_{X \vee Y}$,}
& = [0_{X \wedge Y} \vee 0_{X \vee Y}, 1_{X \wedge Y} \wedge 1_{X \vee Y}] \\
& = [0_{X \wedge Y}, 1_{X \wedge Y}] \cap [0_{X \vee Y}, 1_{X \vee Y}] \\
& = (X \wedge Y) \cap (X \vee Y)
\end{alignat*}
\end{proof}

\begin{lemma} \label{lem:block-meet-join}
For blocks $X$ and $Y$,
\begin{enumerate}
\item[(1)]
$X \cap Y \subset X \vee Y$,
\item[(2)]
$X \cap Y \subset X \wedge Y$,
\item[(3)]
$X \cap Y \cap (X \vee Y) = (X \wedge Y) \cap (X \vee Y)$, and
\item[(4)]
$X \cap Y \cap (X \wedge Y) = (X \vee Y) \cap (X \wedge Y)$,%
\footnote{Statement (4) is dual statement (3), but that is
hard to see from the formulas: Duality on $L$
interchanges meet ($\wedge$) and join ($\vee$) as expected, but it leaves
the intersection of blocks ($\cap$) unchanged.}
\end{enumerate}
where $X \vee Y$ and $X \wedge Y$ are the join and meet operations
in $S$.
\end{lemma}
\begin{proof}
\leavevmode

These are all immediate from lem.~\ref{lem:block-mj}.
\end{proof}

\begin{theorem}\flagstart$\dagger$ \label{th:dissection}
For any modular \lffc lattice $L$,
\begin{enumerate}
\item the skeleton lattice $S$ (of blocks),
\item the dissection tolerance $\relgamma$,
\item the blocks $S$, indexed by themselves (as elements of $S$), and
\item the connections $\P{\bullet}{\bullet}$
\end{enumerate}
form a \mcs.
\end{theorem}
\begin{proof}
\leavevmode

We prove the axioms:

Regarding \ref{MCf}: By lem.~\ref{lem:skeleton-lffc}, $S$ is a \lffc lattice.

Regarding \ref{MCg}: By lem.~\ref{lem:blocks-modular}, each block in $S$ is
finite, modular, and complemented.

Regarding \ref{MCe}: By def.~\ref{def:block-phi}, each $\P{B}{C}$ is the
identity map on the lattice $\F{B}{C} = \I{B}{C}$, and so is a lattice
isomorphism.  By lem.~\ref{lem:block-interval}, each $\F{B}{C}$ is a
filter of $B$ and and $\I{B}{C}$ is an ideal of $C$.

Regarding \ref{MCa}: This is proved as lem.~\ref{lem:block-identity}.

Regarding \ref{MCh}: If $B \leq C \leq D$ in $S$ and
$\I{B}{C} \cap \F{C}{D} \neq \zeroslash$, then
by def.~\ref{def:block-phi} and lem.~\ref{lem:block-mid-intersect},
$\zeroslash \neq (B \cap C) \cap (C \cap D) = B \cap D$, implying
$B \relgamma D$.

Regarding \ref{MCc}: This follows from lem.~\ref{lem:block-mid-intersect}
and that all of the $\P{\bullet}{\bullet}$ are identity maps.

Regarding \ref{MCd}: This is proved as lem.~\ref{lem:block-meet-join}.

Regarding \ref{MCb1} and \ref{MCb2}: These are proved as
lem.~\ref{lem:block-cover}.
\end{proof}

\begin{definition} \label{def:dissection}
We define the \mcs $(S, \gamma, S, \phi)$ described in
th.~\ref{th:dissection} as the \emph{dissection} of $L$.
\end{definition}

\begin{theorem} \label{th:block-minimum}
The lattice $L$ has a minimum element iff its skeleton lattice $S$ has
a minimum element.
\end{theorem}
\begin{proof}
\leavevmode

Regarding $\Rightarrow$:
Assume $L$ has a minimum element $\hatzero_L$.
Then $X = [\hatzero_L, (\hatzero_L)^*]$ is a block with minimum element
$\hatzero_L$.  By lem.~\ref{lem:diss-skel}(2), the fact that
$\hatzero_L$ is the minimum of $L$ implies $X$ is $\leq$
all blocks in $S$.

Regarding $\Leftarrow$:
Assume $S$ has a minimum element $\hatzero_S$.
Every element $c$ of $L$ is a member of the block $[{c^*}_*, c^*]$,
By lem.~\ref{lem:diss-skel}(2), the fact that $\hatzero_S$ is the
minimum of $S$ implies that $0_{\hatzero_S} \leq {c^*}_*$.
Thus $0_{\hatzero_S} \leq {c^*}_* \leq $c, showing that
$\hatzero_L = 0_{\hatzero_S}$.
\end{proof}

\paragraph{Blocks containing an element}

\begin{definition} \label{def:delta}
For $x \in L$, we define $\nabla_x = [x_*,{x_*}^*]$ and
$\Delta_x = [{x^*}_*,x^*]$.
\end{definition}

\begin{lemma} \label{lem:x-in-delta-nabla}
For $x \in L$, $\nabla_x$ and $\Delta_x$ are blocks $\in S$,
$x \in \nabla_x$ and $x \in \Delta_x$.
\end{lemma}
\begin{proof}
\leavevmode

By lem.~\ref{lem:star-props}(c)(c\supdelta) and
lem.~\ref{lem:diss-skel}(10), $[{x^*}_*, x^*]$ and $[x_*, {x_*}^*]$
are both blocks.
By lem.~\ref{lem:star-props}(b)(b\supdelta),
$x \in [{x^*}_*, x^*]$ and $x \in [x_*, {x_*}^*]$.
\end{proof}

\begin{definition} \label{def:Gamma}
$\Gamma_x$ for $x \in L$ is defined as the set of blocks containing $x$.
\end{definition}

Note that $\Gamma_x$ in a dissection is the parallel of $\Pi_x$ in a gluing.

\begin{lemma} \label{lem:Gamma-convex}
For $x \in L$, $\Gamma_x$ is convex.
\end{lemma}
\begin{proof}
\leavevmode

Suppose $B, D \in \Gamma_x$ and $C \in S$ with $B \leq C \leq D$.
Then $0_B \leq x \leq 1_B$, $0_D \leq x \leq 1_D$,
$0_B \leq 0_C \leq 0_D$ and $1_B \leq 1_C \leq 1_D$.
From these follows $0_C \leq 0_D \leq x \leq 1_B \leq 1_C$,
so $C \in \Gamma_x$.
\end{proof}

\begin{lemma} \label{lem:Gamma-extremes} \cite{Haim1991a}*{sec.~1}
$\nabla_x$ is the lowest block in $\Gamma_x$ and
$\Delta_x$ is the highest block in $\Gamma_x$.%
\footnote{Thus there is this mnemonic for $\nabla_\bullet$ and
$\Delta_\bullet$:  $\nabla_x$ points \emph{downward} and is
the \emph{lowest} block containing $x$, and conversely
$\Delta_\bullet$ points \emph{upward} and is the \emph{highest}
block containing $x$.}
\end{lemma}
\begin{proof}
\leavevmode

By lem.~\ref{lem:x-in-delta-nabla}, both $\nabla_x$ and $\Delta_x$ are
$\in \Gamma_x$.

Let $B \in \Gamma_x$ and so $x \in B$,
$x \leq 1_B$, so $x_* \leq (1_B)_* = 0_B$.
Thus $\nabla_x = [x_*,{x_*}^*] \leq B$ by def.~\ref{def:block-ops}.
Similarly, $x \geq 0_B$, so $x^* \geq (0_B)^* = 1_B$.
Thus $\Delta_x = [{x^*}_*,x^*] \geq B$ by def.~\ref{def:block-ops}.
\end{proof}

\begin{lemma} \label{lem:Gamma-finite}
For $x \in L$, $\Gamma_x$ is the interval $[\nabla_x, \Delta_x]$ of
$S$, which is finite.
\end{lemma}
\begin{proof}
\leavevmode

By lem.~\ref{lem:Gamma-extremes}, $\Gamma_x$ has maximum element
$\Delta_x$ and minimum element $\nabla_x$ and by
lem.~\ref{lem:Gamma-convex}, $\Gamma_x$ is convex.
Thus $\Gamma_x$ is the interval $[\nabla_x, \Delta_x]$ of $S$.
By
lem.~\ref{lem:skeleton-lffc}, $S$ is locally finite, so $\Gamma_x$ is finite.
\end{proof}

\paragraph{Dissection is natural}
Dissection is ``natural'' in that it is compatible with
isomorphism of lattices and isomorphism of
\mcss (see def.~\ref{def:mcs-iso}).

\begin{theorem}\flagstart$\dagger$ \label{th:diss-natural}
Dissection is compatible with isomorphism of lattices and isomorphism of
\mcss:  If two lattices are isomorphic, their dissections are
isomorphic.
\end{theorem}
\begin{proof}
\leavevmode

Assume we have two lattices, $L$ and $L^\prime$, and an isomorphism
$\chi$ from $L$ to $L^\prime$.

Let $\scrC$ be the dissection of $L$ and be
comprised of skeleton $S$, overlap tolerance $\relgamma$,
blocks $L_\bullet$, connections $\P{\bullet}{\bullet}$,
connection sources
$\F{\bullet}{\bullet}$, and connection targets
$\I{\bullet}{\bullet}$.
Let $\scrC^\prime$ be the dissection of $L^\prime$
and be comprised of skeleton $S^\prime$, overlap
tolerance $\relgamma^\prime$,
blocks $L_\bullet^\prime$, connections
$\P{\bullet}{\bullet}^\prime$, connection sources
$\F{\bullet}{\bullet}^\prime$, and connection targets
$\I{\bullet}{\bullet}^\prime$.
Thus for $\scrC$ we have derived objects $M$, $\sim$, $\kappa$,
$\pi_\bullet$, and its sum $L$,
and for $\scrC^\prime$ derived objects $M^\prime$, $\sim^\prime$,
$\kappa^\prime$, $\pi^\prime_\bullet$ and its sum $L^\prime$.

Define $\chi_S$ from $S$ to $S^\prime$ that maps a block
$x = [0_x,1_x]\in S$ with $0_x, 1_x \in L$ into
$[\chi\,0_x, \chi\,1_x] \in S^\prime$.
Since $\chi$ is a lattice isomorphism,
$[\chi\,0_x, \chi\,1_x]$ is a block of $L^\prime$.
Conversely, $\chi_S^{-1}\,x^\prime =
\chi_S^{-1}\,[0_{x^\prime}, 1_{x^\prime}] =
[\chi^{-1}\,0_{x^\prime}, \chi^{-1}\,1_{x^\prime}]$
so $\chi_S^{-1}$ maps the blocks of $L^\prime$ to blocks of $L$, showing
that $\chi_S$ is a bijection.
Since $\chi$ is a lattice isomorphism, def.~\ref{def:block-ops} shows
that $\chi_S$ is a lattice isomorphism from $S$ to $S^\prime$.

For any $x \in S$, define $\chi_{Bx}$ on $L_x$ by
$\chi_{Bx}\,a = \chi\,a$ for all $a \in L_x$.
Then given $a \in L_x$,
\begin{alignat*}{2}
a & \in [0_x, 1_x] \\
\interject{because $\chi$ is an isomorphism,}
\chi\,a & \in [\chi\,0_x, \chi\,1_x] \\
\interject{by the definition of $\chi_S$,}
\chi\,a & \in \chi_S\,[0_x, 1_x] = \chi_S\,x \\
\interject{by the definition of $\chi_{Bx}$,}
\chi_{Bx}\,a & \in \chi_S\,x \\
\interject{since the blocks in $S^\prime$ are their own indexes
(th.~\ref{th:dissection}),}
\chi_{Bx}\,a & \in L_{\chi_S\,x}^\prime
\end{alignat*}
showing that $\chi_{Bx}$ maps $L_x$ to $L_{\chi_S\,x}^\prime$.

Similarly, $\chi_{Bx}^{-1}\,a^\prime = \chi^{-1}\,a^\prime$ and
$\chi^{-1}$ maps $L_{\chi_S\,x}^\prime$ to $L_x$.
Those facts, combined with the fact that $\chi$ is a lattice
isomorphism, shows that $\chi_{Bx}$ is a lattice
isomorphism from $L_x$ to $L_{\chi_S\,x}^\prime$.

Since $\relgamma$ on $S$ and $\relgamma^\prime$ on $S^\prime$ are
defined as non-null intersection of blocks (def.~\ref{def:diss-tol})
and $\chi_S$ on $S$ is defined as
applying the bijection $\chi$ on the elements of its argument,
for $x, y \in S$,
$x \relgamma y$ iff $\chi_S(x) \relgamma^\prime \chi_S(y)$.

Given $x \leq_\gamma y \in S$, by lem.~\ref{lem:block-interval},
$\F{x}{y} = [0_y, 1_x]$, so
\begin{alignat*}{2}
\chi_{Bx}(\F{x}{y})
& = \chi_{Bx}([0_y, 1_x]) \\
\interject{by the definition of $\chi_{Bx}$,}
& = \chi([0_y, 1_x]) \\
\interject{applying $\chi$ to each member of $[0_y, 1_x]$,}
& = [\chi\,0_y, \chi\,1_x] \\
\interject{by the definition of $\chi_S$,}
& = [0_{\chi_S\,y}, 1_{\chi_S\,x}] \\
\interject{by lem.~\ref{lem:block-interval},}
& = \F{\chi_S\,x}{\chi_S\,y}^\prime
\end{alignat*}
Thus showing that $\F{\chi_S(x)}{\chi_S(y)}^\prime = \chi_{Bx}(\F{x}{y})$.

Similarly, we show that $\I{\chi_S(x)}{\chi_S(y)}^\prime = \chi_{By}(\I{x}{y})$.

Since the connections $\P{\bullet}{\bullet}$ in dissections are
identity maps (def.~\ref{def:block-phi}), for all $a \in \F{x}{y}$,
$\chi_{By}(\P{x}{y}(a)) = \chi_{By}(a) =
\P{\chi_S(x)}{\chi_S(y)}^\prime(\chi_{Bx}(a))$, or equivalently,
$\chi_{By} \circ \P{x}{y} =
\P{\chi_S(x)}{\chi_S(y)}^\prime \circ \chi_{Bx}|_{\F{x}{y}}$.%

These facts together show that $\chi_S$ and $\chi_{B\bullet}$ comprise
a \mcs isomorphism (def.~\ref{def:mcs-iso} from $\scrC$ to $\scrC^\prime$.
\end{proof}

\section{Dissection of a gluing}

In this section, we prove that if $\scrC^\prime$ is the dissection of the
sum lattice $L$ of a \mcs $\scrC$, then $\scrC^\prime$ is isomorphic to
$\scrC$.

We choose $\scrC$ to be an arbitrary \mcs.
We define $L$ to be the sum (def.~\ref{def:sum}) of $\scrC$.
We define $\scrC^\prime$ to be the dissection
(def.~\ref{def:dissection}) of $L$.

From $\scrC$, we define its components:
the skeleton $S$, the overlap tolerance $\relgamma$, and
the blocks $L_\bullet$,
From $\scrC$, we define its derived objects:
the sum $L$, the canonical projections (def.~\ref{def:pi})
$\pi_\bullet$,
the blocks $\Lambda_\bullet$ (def.~\ref{def:Lam}) of $L$, and
the elements $0_\bullet$ and $1_\bullet$ (def.~\ref{def:Lam}) in $L$.

From $L$, we define its derived objects:
the set of blocks (def.~\ref{def:block}) $S^\prime$,
the dissection tolerance (def.~\ref{def:diss-tol})
$\relgamma^\prime$, and
the dissection $\scrC^\prime$.

\begin{definition} \label{def:chi-S}
We define $\chi_S$:  for any $x \in S$, $\chi_S(x)$ is the interval
$\Lambda_x = [0_x, 1_x]$ of $L$.
\end{definition}

\begin{lemma}\flagstart$\dagger$ \label{lem:Lam-is-block}
For each $x \in S$, $\Lambda_x$ is a block of $L$.
\end{lemma}
\begin{proof}
\leavevmode

$\hatone_{L_x}$ is the join of the atoms of $L_x$.
$\Lambda_x$ is isomorphic to $L_x$ and the atoms of $L_x$ map via $\pi_x$
to elements that cover $0_x$.
Thus, ${0_x}^* \geq 1_x$.

Suppose there is an element $a \in L$ covering $0_x$ that is not in
$\Lambda_x$.
Then by lem.~\ref{lem:covers-xfer} there is a $y \in S$ such that
$0_x$ and $a$ are in $\Lambda_y$.
Then $0_x \in \Lambda_y$, so $0_x \geq 0_y$.
By lem.~\ref{lem:(12)-monotone}, $x \geq y$.
But if $x = y$, then $a \in \Lambda_x$, which is contrary to hypothesis,
so we must have $x > y$.
By lem.~\ref{lem:(12)-monotone}, $1_x \geq 1_y$.
Then, $0_x \leq a \leq 1_y \leq 1_x$, and since $\Lambda_x = [0_x,1_x]$,
$a \in \Lambda_x$, contrary to hypothesis.
Thus all $a \gtrdot 0_x$ are $\in \Lambda_x$ and ${0_x}^* \leq 1_x$.

So we have shown $(0_x)^* = 1_x$.
Dually, we show $(1_x)_* = 0_x$.
Thus by lem.~\ref{lem:diss-skel}(1), $\Lambda_x = [0_x,1_x]$ is a block.
\end{proof}

\begin{lemma}\flagstart$\dagger$ \label{lem:block-is-Lam}
Every block of $L$ is a $\Lambda_x$ for some $x \in S$.
\end{lemma}
\begin{proof}
\leavevmode

Let $[a, b]$ be a block of $L$.
Assuming $L$ is not the trivial lattice, $a \neq b$.
By lem.~\ref{lem:diss-skel}(1), $a^* = b$ and $b_* = a$.

Let $c_1, c_2, \ldots, c_n$ be the atoms of $[a,b]$, that is, the elements of
$[a, b]$ that cover $a$.
(This is a finite set because $L$ has finite covers.)
For each $c_i$, by lem.~\ref{lem:covers-xfer}, there exists $x_i \in S$ such that
$a, c_i \in \Lambda_{x_i}$.
Define $x = \bigvee_i x_i$.
By lem.~\ref{lem:MCd-for-Lam}, $a \in \Lambda_x$.
By lem.~\ref{lem:(1)-for-Lam}, $\Lambda_{x_i} \cap \Lambda_x$ is a
filter of $\Lambda_x$, and so $c_i \in \Lambda_x$.
Because $\Lambda_x$ is a sublattice of $L$, $\bigvee_i c_i \in \Lambda_x$,
so $b = a^* = \bigvee_i c_i \in \Lambda_x$.

This shows that $[a,b] \subset \Lambda_x$.
By lem.~\ref{lem:Lam-is-block}, $\Lambda_x$ is a block,
and since all blocks are maximal under containment, $[a,b] = \Lambda_x$.
\end{proof}

\begin{lemma}\flagstart$\dagger$
$\chi_S$ is a bijection from $S$ to $S^\prime$.
\end{lemma}
\begin{proof}
\leavevmode

By lem.~\ref{lem:Lam-is-block}, $\chi_S$ maps $S$ to $S^\prime$.
By lem.~\ref{lem:Lam-incomp}, $\chi_S$ is one-to-one.
By lem.~\ref{lem:block-is-Lam}, $\chi_S$ is onto.
Thus $\chi_S$ is a bijection from $S$ to $S^\prime$.
\end{proof}

\begin{lemma} \label{lem:gamma-iso}
For $x, y \in S$, $x \relgamma y$ iff
$\chi_S(x) \relgamma^\prime \chi_S(y)$.
\end{lemma}
\begin{proof}
\leavevmode

Regarding $\Rightarrow$:
Suppose $x \relgamma y$.  By lem.~\ref{lem:Lam-overlap-gamma},
$\Lambda_x \cap \Lambda_y \neq \zeroslash$.
Thus by def.~\ref{def:chi-S}, $\chi_S(x) \cap \chi_S(y) \neq \zeroslash$ and by
def.~\ref{def:diss-tol}, $\chi_S(x) \relgamma^\prime \chi_S(y)$.

Regarding $\Leftarrow$:
Suppose $\chi_S(x) \relgamma^\prime \chi_S(y)$.
By def.~\ref{def:diss-tol} and~\ref{def:chi-S},
$\Lambda_x \cap \Lambda_y \neq \zeroslash$.
By lem.~\ref{lem:Lam-overlap-gamma}, $x \relgamma y$.
\end{proof}

\begin{definition} \label{def:chi-B}
For $x \in S$, we define
$\chi_{Bx}: L_x \rightarrow \Lambda_x: y \mapsto \pi_x\,y$,
which is a lattice isomorphism from $L_x$ to $\Lambda_x$.
\end{definition}

\begin{theorem} \label{th:dissection-of-sum}
The dissection $\scrC^\prime$ of the sum $L$ of the \mcs $\scrC$
is isomorphic to $\scrC$.
\end{theorem}
\begin{proof}
\leavevmode

We will prove that $\chi_S$ and $(\chi_{Bx})_{x \in S}$ as defined by
def.~\ref{def:chi-S} and~\ref{def:chi-B} are a \mcs isomorphism
(def.~\ref{def:mcs-iso}) from $\scrC$ to $\scrC^\prime$.
The needed characteristics of a \mcs isomorphism are proved as
follows:
\begin{enumerate}
\item This is proven by lem.~\ref{lem:gamma-iso}.
\item This is proven by def.~\ref{def:diss-tol},
def.~\ref{def:block-phi}, and lem.~\ref{lem:(1)-for-Lam}.
\item This is proven by def.~\ref{def:chi-B} and lem.~\ref{lem:(1)-for-Lam}.
\end{enumerate}
\end{proof}

\section{Gluing of a dissection}

In this section, we prove that if $L^\prime$ is the sum of the
dissection $\scrC$ of a modular, \lffc lattice $L$, then $L^\prime$ is
isomorphic to $L$.

We choose $L$ to be an arbitrary modular, \lffc lattice.
We define $\scrC$ to be the \mcs which is the
dissection (def.~\ref{def:dissection}) of $L$.
We define $L^\prime$ to be the sum (def.~\ref{def:sum}) of $\scrC$.

From the dissection of $L$, we define its derived objects:
the set of blocks (def.~\ref{def:block}) $S$,
the operations $\bullet_*$, $\bullet^*$ (def.~\ref{def:star}),
the blocks $\Delta_\bullet$, $\nabla_\bullet$ (def.~\ref{def:delta}),
and
the sets of blocks $\Gamma_\bullet$ (def.~\ref{def:Gamma}).

From $\scrC$, we define its components:
the family of blocks $L_\bullet^\prime$, which is the family
of blocks (def.~\ref{def:block-phi}) of $L$ indexed by themselves (as
elements of $S$); and
the connections $\P{\bullet}{\bullet}^\prime$ and their
sources and targets $\F{\bullet}{\bullet}^\prime$,
$\I{\bullet}{\bullet}^\prime$.
From $\scrC$, we define its derived objects:
$M^\prime$, $\sim^\prime$, (def.~\ref{def:sim}),
$\kappa^\prime$ (def.~\ref{def:kappa}),
the set of ascending sequences
(def.~\ref{def:asc-seq-new}),
and the sum $L^\prime$.

\begin{definition} \label{def:chi}
We define
$\chi:L \rightarrow L^\prime: a \mapsto \kappa^\prime\,(\nabla_a, a)$.
\end{definition}

That is, $\chi$ takes an element $a \in L$, derives the block of the
dissection $\nabla_a \in S$, then maps the element/block pair through
the gluing using $\kappa^\prime$ to an element of $L^\prime$.

\begin{lemma}
$\chi$ is well-defined.
\end{lemma}
\begin{proof}
\leavevmode

Choose $a \in L$.
By lem.~\ref{lem:Gamma-extremes}, $\nabla_a$ is a block, and is
the lowest block in $\Gamma_a$, which is the set of blocks containing
$a$.
Thus, $a \in \nabla_a$, and so $\kappa^\prime$ maps $(\nabla_a, a)$
into an element of the sum $L^\prime$.
\end{proof}

\begin{lemma} \label{lem:chi-any-block}
Given any $x, y \in S$, $p \in L^\prime_x$, and $q \in L^\prime_y$, then
$\kappa^\prime\,(x, p) = \kappa^\prime\,(y, q)$ (in $L^\prime$) iff $p = q$ (in $L$).
\end{lemma}
\begin{proof}
\leavevmode

First we note that if $p = q$ is $\in L^\prime_x, L^\prime_y$, then
$p = q \in L^\prime_{x \vee y}$ by lem.~\ref{lem:block-meet-join}(1).
That shows that
$p \in \F{x}{x \vee y}^\prime = L^\prime_x \cap L^\prime_{x \vee y}$ and
$q \in \F{y}{x \vee y}^\prime = L^\prime_y \cap L^\prime_{x \vee y}$ by
def.~\ref{def:block-phi}.

The following statements are all equivalent:
\begin{alignat*}{2}
\kappa^\prime\,(x, p) & = \kappa^\prime\,(y, q) \\
\interject{by def.~\ref{def:kappa},}
(x, p) & \sim^\prime (y, q) \\
\interject{by def.~\ref{def:sim},}
\P{x}{x \vee y}^\prime\,p & = \P{y}{x \vee y}^\prime\,q \textup{ and }
p \in \F{x}{x \vee y}^\prime \textup{ and }
q \in \F{y}{x \vee y}^\prime \\
\interject{by def.~\ref{def:block-phi}, the $\P{\bullet}{\bullet}^\prime$ are
identity maps;
and in the reverse direction,
by the above derivation that $p = q$ implies
$p \in \F{x}{x \vee y}^\prime = L^\prime_x \cap L^\prime_{x \vee y}$ and
$q \in \F{y}{x \vee y}^\prime = L^\prime_y \cap L^\prime_{x \vee y}$,}
p = q
\end{alignat*}
\end{proof}

This shows that our choice of $\nabla_a$ as the block in
def.~\ref{def:chi} was arbitrary; any other block in $\Gamma_a$
(i.e., any other block containing $a$) could be used in the definition
without changing $\chi\,a$.

\begin{lemma} \label{lem:chi-bij}
$\chi$ is a bijection from $L$ to $L^\prime$.
\end{lemma}
\begin{proof}
\leavevmode

Regarding that $\chi$ is one-to-one:
Assume that $a, a^\prime \in L$ and $\chi\,a = \chi\,a^\prime$.
Then $\kappa^\prime\,(\nabla_a, a) = \kappa^\prime\,(\nabla_{a^\prime}, a^\prime)$.
By the definition of $\kappa^\prime$, that requires that
$(\nabla_a, a) \sim^\prime (\nabla_{a^\prime}, a^\prime)$ which implies that
$\P{\nabla_a}{\nabla_a \vee \nabla_{a^\prime}}^\prime\, a =
\P{\nabla_{a^\prime}}{\nabla_a \vee \nabla_{a^\prime}}^\prime\, a^\prime$.
But since the $\P{\bullet}{\bullet}^\prime$ are identity maps, that means that
$a = a^\prime$.

Regarding that $\chi$ is onto:
Choose any $a^\prime \in L^\prime$.
Since $L^\prime$ is the set of values of $\kappa^\prime$, there exists
a $B \in S$ and an $a \in B$ such that $\kappa^\prime\,(B, a) = a^\prime$.
By def.~\ref{def:block}, $B$ is a block of $L$ and $a \in L$,
so $B \in \Gamma_a$.
Then by lem.~\ref{lem:chi-any-block},
$a^\prime = \kappa^\prime\,(B, a) = \chi\,a$.

Thus, $\chi$ is a bijection from $L$ to $L^\prime$.
\end{proof}

\begin{theorem} \label{th:sum-of-dissection}
The sum $L^\prime$ of the \mcs $\scrC$ which is
the dissection of a modular, \lffc lattice $L$ is
is isomorphic to $L$.
\end{theorem}
\begin{proof}
\leavevmode

We will show that $\chi$ is a lattice isomorphism from $L$ to
$L^\prime$.  In lem.~\ref{lem:chi-bij} we have already shown that
$\chi$ is a bijection, so what remains is to show that for any
$a, b \in L$, $a \leq b$ in $L$ iff $\chi\,a \leq \chi\,b$ in
$L^\prime$.

Regarding $\Rightarrow$:

Assume $a \leq b$ in $L$.
Choose a saturated chain in $L$ from $a$ to $b$:
\begin{equation*}
a = c_0 \lessdot c_1 \lessdot c_2 \lessdot \cdots
\lessdot c_{n-1} \lessdot c_n = b. \tag{a}
\end{equation*}
For $1 \leq i \leq n$, define $d_i = \nabla_{c_i} \in S$.
Since the $L^\prime_\bullet$ are themselves the blocks of $L$, which are
the members of $S$, each $d_i$ is also one of the $L^\prime_\bullet$, specifically
$d_i \in L^\prime_{d_i}$.

Since $c_{i-1} \leq c_i$ in $L$,
by lem.~\ref{lem:star-props}(a\supdelta), $(c_{i-1})_* \leq (c_i)_*$,
then by def.~\ref{def:block-ops},
$[(c_{i-1})_*, {(c_{i-1})_*}^*] \leq [(c_i)_*, {(c_i)_*}^*]$ in $S$,
and then by def.~\ref{def:delta},
$d_{i-1} = \nabla_{c_{i-1}} \leq \nabla_{c_i} = d_i$ in $S$.
Thus, in $S$,
\begin{equation*}
d_1 \leq d_2 \leq \cdots \leq d_{n-1} \leq d_n. \tag{b}
\end{equation*}

That $c_i \in d_i = \nabla_{c_i}$ is immediate.
Thus $(d_i, c_i) \in M^\prime$ and
$\chi\,c_i = \kappa^\prime\,(d_i, c_i)$ exists in $L^\prime$.

By def.~\ref{def:star}, $c_{i-1} \lessdot c_i$ implies $(c_i)_* \leq c_{i-1}$.
Then by lem.~\ref{lem:star-props}(b),
$c_{i-1} < c_i \leq {(c_i)_*}^*$.
These show
$c_{i-1} \in [(c_i)_*, {(c_i)_*}^*] = \nabla_{c_i} = d_i$.
Thus $(d_i, c_{i-1}) \in M^\prime$ and $\kappa^\prime\,(d_i, c_{i-1})$ exists in
$L^\prime$.
By lem.~\ref{lem:chi-any-block},
$\kappa^\prime\,(d_i, c_{i-1}) = \kappa^\prime\,(d_{i-1}, c_{i-1})$.

Because $c_{i-1}, c_i \in d_i$ and $c_{i-1} \leq c_i$
(in the order of $L^\prime_{d_i}$, which is inherited from the order of $L$),
by def.~\ref{def:leq-sub-x},
$\chi\,c_{i-1} = \kappa^\prime\,(d_{i-1}, c_{i-1})  = \kappa^\prime\,(d_i, c_{i-1})
\leq_{d_i} \kappa^\prime\,(d_i, c_i) = \chi\,c_i$ in $L^\prime$.

Thus, in $L^\prime$,
\begin{equation*}
\chi(a) = \chi(c_0) \leq_{d_1} \chi(c_1) \leq_{d_2} \chi(c_2) \leq_{d_3} \cdots
\leq_{d_{n-1}} \chi(c_{n-1}) \leq_{d_n} \chi(c_n) = \chi(b). \tag{c}
\end{equation*}

By def.~\ref{def:asc-seq-new}, (c) and (b) show that
$(d_\bullet, \chi\,c_\bullet)$ is
an ascending sequence (in $\scrC$), and so by def.~\ref{def:L-leq},
$\chi\,a \leq \chi\,b$.

Regarding $\Leftarrow$:

Assume $a, b \in L$ and $\chi\,a \leq \chi\,b$ in $L^\prime$.
Then by def.~\ref{def:L-leq}, there is an ascending sequence (in $\scrC$)
$(z_\bullet, e_\bullet)$ so that in $L^\prime$
\begin{equation*}
\chi(a) = e_0 \leq_{z_1} e_1 \leq_{z_2} e_2 \leq_{z_3} \cdots
\leq_{z_{n-1}} e_{n-1} \leq_{z_n} e_n = \chi(b). \tag{d}
\end{equation*}
and in $S$
\begin{equation*}
z_1 \leq z_2 \leq z_3 \leq \cdots \leq z_{n-1} \leq z_n. \tag{e}
\end{equation*}
Define $z_0 = \nabla_a$ so $a = e_0 \in L^\prime_{z_0}$,
and since $a = e_0 \in L^\prime_{z_1}$ from (d), $z_0 \leq L^\prime_{z_1}$.
For $0 \leq i \leq n$, define $f_i = \pi_{z_i}^{-1}\,e_i \in z_i$.

For $1 \leq i \leq n$, from (d) and the definition of $z_0$,
$e_{i-1} \in L^\prime_{z_{i-1}}$ and $e_{i-1} \in L^\prime_{z_i}$.
By def.~\ref{def:kappa},
$\kappa^\prime\,(z_{i-1}, \pi^{-1}_{z_{i-1}}\,e_{i-1}) = e_{i-1} =
\kappa^\prime\,(z_i, \pi^{-1}_{z_i}\,e_{i-1})$.
By lem.~\ref{lem:chi-any-block},
$\pi^{-1}_{z_{i-1}}\,e_{i-1} = \pi^{-1}_{z_i}\,e_{i-1}$, showing
$f_{i-1} = \pi^{-1}_{z_i}\,e_{i-1}$.

For $1 \leq i \leq n$, from (d),
$e_{i-1} \leq_{z_i} e_i$, which is, by def.~\ref{def:leq-sub-x},
$\pi^{-1}_{z_i}\,e_{i-1} \leq \pi^{-1}_{z_i}\,e_i$,
which by the above shows $f_{i-1} \leq f_i$ in $L^\prime_{z_i}$.
Since $L^\prime_{z_i}$ inherits its order from $L$, $f_{i-1} \leq f_i$ in
$L$.  Thus in $L$
\begin{equation*}
f_0 \leq f_1 \leq f_2 \leq \cdots \leq f_{n-1} \leq f_n. \tag{f}
\end{equation*}

Then
$\kappa^\prime\,(z_0, f_0) = \kappa^\prime\,(z_0, \pi_{z_0}^{-1}\,e_0) =
e_0 = \chi\,a = \kappa^\prime\,(\nabla_a, a)$,
so by lem.~\ref{lem:chi-any-block}, $f_0 = a$.
By lem.~\ref{lem:chi-any-block}, $\kappa^\prime\,(z_n, f_n) = e_n = \chi\,b$
implies $f_n = b$ in $L$.  These facts, with (f), show $a \leq b$ in $L$.
\end{proof}

\section{Conclusion}

\begin{theorem}\flagstart$\dagger$ \label{th:inverses}
The two mappings (1) gluing a \mcs to create a modular, \lffc lattice and
(2) dissecting a modular, \lffc lattice to create a \mcs are mutually
inverse mappings between
the category of isomorphism classes of \mcss and
the category of isomorphism classes of modular, \lffc lattices.
Thus, the two mappings are both bijective.
The the skeleton of the \mcs has a
minimum element iff the corresponding lattice has a minimum element.
\end{theorem}
\begin{proof} \disconnect
That the two processes map between isomorphism classes of
modular, \lffc lattices and isomorphism classes of \mcss is proven in
th.~\ref{th:glue-natural} and~\ref{th:diss-natural}.

That the two processes are inverses is proven in
th.~\ref{th:dissection-of-sum} and~\ref{th:sum-of-dissection}.

That the lattice has a minimum element iff the skeleton has a
minimum element is proven in th.~\ref{th:sum-0}.
\end{proof}

This bijection allows us to informally conflate a lattice and its
corresponding \mcs.  This allows us to consider the elements of the
lattice to be simultaneously elements of the blocks of the \mcs,
to conflate the blocks $L_\bullet$ of the \mcs with the blocks of the
lattice, etc.%
\footnote{We use the notation
``$\leadingcolon X \rightarrow Y:x \mapsto P(x)$''
to denote an anonymous function.}
This resembles the approach in \citeHerr*{sec.~1}{sec.~1}.
\begin{alignat*}{2}
L_x & \longleftrightarrow \Lambda_x \\
x \relgamma y & \longleftrightarrow \Lambda_x \cap \Lambda_y \neq \zeroslash \\
\F{x}{y} & \longleftrightarrow \Lambda_x \cap \Lambda_y \\
\I{x}{y} & \longleftrightarrow \Lambda_x \cap \Lambda_y \\
\P{x}{y} & \longleftrightarrow
\leadingcolon \Lambda_x \cap \Lambda_y \rightarrow \Lambda_x \cap \Lambda_y
: a \mapsto a \\
\pi_x & \longleftrightarrow
\leadingcolon \Lambda_x \rightarrow \Lambda_x : a \mapsto a \\
\kappa\,(x, a) & \longleftrightarrow a \\
0_{L_x} & \longleftrightarrow 0_x \\
\Z{x}{y} & \longleftrightarrow 0_y \\
1_{L_x} & \longleftrightarrow 1_x \\
\I{x}{y} & \longleftrightarrow 1_x \\
\PP{x}{y} & \longleftrightarrow
\leadingcolon \Lambda_x \rightarrow \Lambda_y : a \mapsto a \vee 0_y \\
\PPI{x}{y} & \longleftrightarrow
\leadingcolon \Lambda_y \rightarrow \Lambda_x : a \mapsto a \wedge 1_x
\end{alignat*}

This equivalence is a base for the further explication of the structure of
modular, \lffc lattices.

\vspace{3em}

\end{document}